\documentclass[12pt,draftcls,onecolumn]{IEEEtran}
\IEEEoverridecommandlockouts                              % This command is only
\makeatletter
\let\NAT@parse\undefined
\makeatother

\usepackage{graphicx}% Include figure files
\usepackage{bm}% bold math
\usepackage{epsfig}
\usepackage{amsthm}
\usepackage{amssymb,amsfonts,amsmath}
\usepackage{subfigure}
\usepackage{setspace}
\usepackage{enumerate}
\usepackage[numbers]{natbib}
\usepackage{color}
\usepackage{dsfont}
\usepackage{verbatim}

\newtheorem{theorem}{Theorem}
\newtheorem{lemma}{Lemma}
\newtheorem{remark}{Remark}
\newtheorem{definition}{Definition}
\newtheorem{proposition}{Proposition}
\newtheorem{example}{Example}

\newtheorem{corollary}{Corollary}

\def \b {\beta}
\def \e {\varepsilon}
\def \l {\lambda}
\def \L {\mathcal{L}}
\def \<{\langle}
\def \>{\rangle}

\title{\LARGE \bf
Controllability Canonical Forms of Linear Ensemble Systems
}

\author{Wei Zhang, \IEEEmembership{IEEE Member} and Jr-Shin Li, \IEEEmembership{Senior Member, IEEE}
\thanks{W. Zhang is with the Department of Electrical and Systems Engineering, Washington University,
		St. Louis, MO 63130, USA
        {\tt\small wei.zhang@wustl.edu}}%
\thanks{J.-S. Li is with the Department of Electrical and Systems Engineering, Washington University,
        St. Louis, MO 63130, USA
        {\tt\small jsli@wustl.edu}}%
}

\begin{document}

\maketitle
% \thispagestyle{empty}
% \pagestyle{empty}

%%%%%%%%%%%%%%%%%%%%%%%%%%%%%%%%%%%%%%%%%%%%%%%%%%
\begin{abstract}
Ensemble control, an emerging research field focusing on the study of large populations of dynamical systems, has demonstrated great potential in numerous scientific and practical applications. Striking examples include pulse design for exciting spin ensembles in quantum physics, neurostimulation for relieving neurological disorder symptoms, and path planning for steering robot swarms. However, the control targets in such applications are generally large-scale complex and severely underactuated ensemble systems, research into which stretches the capability of techniques in classical control and dynamical systems theory to the very limit. Even for the simplest class of ensemble systems, that is, time-invariant linear ensemble systems, our understanding is still far from complete. This paper then devotes to advancing our knowledge about controllability of this type of ensemble systems by integrating tools in modern algebra into the technique of separating points developed in our recent work. In particular, we give an algebraic interpretation of the dynamics of linear systems in terms of the action of polynomials on vector spaces, and this leads to the development of the functional canonical form of matrix-valued functions, which can also be viewed as the generalization of the rational canonical form of matrices in linear algebra. Then, leveraging the technique of separating points,  we achieve a necessary and sufficient characterization of uniform ensemble controllability for time-invariant linear ensemble systems as the ensemble controllability canonical form, in which the system and control matrices are in the functional canonical and block diagonal form, respectively. This work successfully launches a new research scheme of adopting and tailoring finite-dimensional methods, such as those in classical linear systems theory and matrix algebra, to tackle control problems involving infinite-dimensional ensemble systems, and hence lays a solid foundation for a more inclusive ensemble control theory targeting at a much broader spectrum of control and learning problems in both scientific research and practice.

\end{abstract}


%======================== Introduction ======================
\section{Introduction}
\label{sec:introduction}
Systems constituting huge amounts, in the limit continuum, of structurally identical dynamic components, called \emph{ensemble systems}, are prevalent in natural sciences and engineering. Targeted coordination and robust manipulation of ensemble systems, referred to as \emph{ensemble control} tasks, then mark an essential step in the study of numerous scientific and practical problems, such as pulse design for nuclear magnetic resonance (NMR) spectroscopy and imaging (MRI) in quantum physics \cite{Glaser98,Li_PRA06,Li_TAC09,Li_PNAS11,Li_NatureComm17}, neurostimulation for treatments of neurological disorders in neurology and neurosurgery \cite{Hochberg06,Ching2013b,Li_TAC13,Kafashan2015}, synchronization analysis for rhythmic networks in network science \cite{Rosenblum04,Nakao07,Kiss07,Li_NatureComm16}, path planning for robot swarms in robotics \cite{Becker12}. These emerging applications immediately spot a novel research trend in control theory and engineering targeting at large-scale ensemble systems. However, variation of dispersion parameters is a common phenomenon that occurs in such ensembles revealing discrepancies in the dynamics of their individual units, which, together with the huge size, discloses the severely underacturated nature of ensemble systems. Consequently, ensemble control problems present significant challenges to control theorists and engineers, and are beyond the capability of techniques in classical control theory, e.g., design of state observers and feedback control inputs.  

Driven by the hope for a more inclusive and general control theory that provides innovative thinking about those seemingly inaccessible ensemble control problems, in this decade, considerable efforts have been devoted to expanding the repertoire of tools in and the scope of classical control and dynamical systems theory for technically increasing the accessibility of control and learning tasks involving ensemble systems. These works greatly advance our understanding of fundamental properties, particularly, controllability and observability, of ensemble systems from interdisciplinary perspectives, including differential geometry, Lie theory, representation theory, approximation theory, functional analysis, complex analysis, and probability theory and statistics \cite{Li_TAC09,Li_TAC11,Schonlein13,Helmke14,Li_TAC16,Zeng_16_moment,zeng2016tac,Dirr16,Li_SCL18,XD_Chen19,Li_SICON20}. 

Indubitably, time-invariant linear ensemble systems are the simplest class of ensemble systems, but our knowledge about them is still in an elementary level. The purpose of this paper is then to enhance our understanding of the concept of uniform ensemble controllability for such systems from the perspective of modern algebra. To this end, we initiate the investigation by an algebraic interpretation of the dynamics of linear systems in terms of the action of polynomials on vector spaces, which directly leads to the representation of the system matrices in the rational canonical form. Then, leveraging the technique of separating points established in our recent work \cite{Li_SICON20}, we develop the notion of \emph{functional canonical form} for matrix-valued functions so that uniform ensemble controllability of linear ensemble systems guarantees the feasibility of transforming these systems into an \emph{ensemble controllability canonical form}, in which the system and control matrices are in the functional canonical and block diagonal form, respectively, and vice versa. This, on the other hand, also gives an algebraic characterization of the notion of separating points, and equivalently, necessary and sufficient conditions for both ensemble and classical controllability of linear systems. 

More importantly, this work unifies the concept of separating points for finite-dimensional classical linear systems and infinite-dimensional linear ensemble systems, which further highlights the role of the separating point technique in exploring the possibility of tackling infinite-dimensional ensemble control tasks by utilizing well-established finite-dimensional methods, especially, those in classical linear systems theory and matrix algebra. Conrrespondingly, it also makes a substantial contribution to surmounting the technical obstacles on the way towards a comprehensive ensemble control theory. 

The paper is organized as follows. In the next section, we rigorously introduce the notions of ensemble systems and ensemble control, then briefly review the technique of separating points for analyzing uniform ensemble controllability of linear systems developed in \cite{Li_SICON20}. In Section \ref{sec:canonical}, we introduce the rational canonical form of matrices from a dynamical system viewpoint and illustrate the idea of separating points in the study of controllability for classical linear systems by using the rational canonical form, which represents a major step towards the unification of the separating point technique for classical and ensemble linear systems. Next, following from the foundation laid in Section \ref{sec:canonical}, Section \ref{sec:canoncial_ensemble} devotes to the development of the functional canonical from of matrix-valued functions for algebraically characterizing the notion of separating points in the ensemble case, which directly gives rise to an ensemble controllability canonical form for time-invariant linear ensemble systems.

%======================== Preliminaries ======================
\section{Concepts of Separating Points for Ensemble Controllability}
\label{sec:preliminaries}
%\cb{
Problems concerning the control of an ensemble of linear systems have been extensively studied \cite{Li_TAC09,Li_TAC11,Schonlein13,Helmke14,Li_TAC16,Zeng_16_moment,Dirr16,Li_SCL18}. An important recent finding was the link of the notion of separating points in abstract algebra with fundamental properties of ensemble systems. In particular, this discovery has led to the derivation of a necessary and sufficient condition for uniform ensemble controllability of linear ensemble systems \cite{Li_SICON20} and conditions under which classical and ensemble controllability are equivalent \cite{Li_SICON21_LieGroup,Li_CDC21}. In this section, we give an essential review of the ideas of separating points, which lays the foundation for this work. 

%work, we leverage this recent development to establish 

%also lays the foundation for this work. The purpose of this section is then to give a brief review of this technique, together with introducing some basic concepts in ensemble control theory.
%}
% Ensemble control problems concerning linear systems have been extensively investigated in this decade \cite{Li_TAC11,Schonlein13,Helmke14,Li_TAC16,Zeng_16_moment,Dirr16,Li_SCL18}. The most recent advancement is the technique of separating points, which leads to a necessary and sufficient uniform ensemble controllability condition for linear ensemble systems \cite{Li_SICON20} and also lays the foundation for this work. The purpose of this section is then to give a brief review of this technique, together with introducing some basic concepts in ensemble control theory.

%========================= Ensemble controllability ===============
\subsection{Ensemble Systems and Ensemble Controllability}
An \emph{ensemble  system} is a parameterized family of dynamical systems evolving on a common manifold $M$ in the form of
\begin{align}
\label{eq:ensemble_general}
\frac{d}{dt}x(t,\beta)=F(\beta,x(t,\beta),u(t)),
\end{align}
where $\beta$ is the parameter varying on $\Omega\subseteq\mathbb{R}^d$, $u(t)\in\mathbb{R}^m$ is the control input, and $F(\beta,\cdot,u(t))$ is a vector field on $M$ for each fixed $\beta$ and $u(t)$. The state space of the ensemble system in \eqref{eq:ensemble_general} is then a space of $M$-valued functions defined on the parameter space $\Omega$, denoted by $\mathcal{F}(\Omega,M)$. In many practical problems, $\Omega$ is an infinite set so that $\mathcal{F}(\Omega,M)$ is an infinite-dimensional manifold. An \emph{ensemble control task} is to steer this infinite-dimensional system in \eqref{eq:ensemble_general} between desired states, i.e., functions, in $\mathcal{F}(\Omega,M)$. However, due to the large size and underactuated nature of the ensemble, it is generally impossible to obtain comprehensive measurements of each individual system, which forces the control signal $u$ to be a parameter-independent open-loop input. Consequently, ensemble control tasks greatly challenge and are also beyond the scope of classical control theory. For example, control inputs that steer such a system to a desired target state may not exist. Therefore, to facilitate analysis and control of ensemble systems, we extend the notion of controllability from the exact sense to the approximate sense.

%======================= Definition 1 ============================
\begin{definition}[Ensemble Controllability]
	\label{def:controllability}
	An ensemble system as in \eqref{eq:ensemble_general} is said to be ensemble controllable on $\mathcal{F}(\Omega,M)$, if for any $\varepsilon>0$ and starting with any initial state $x_0\in\mathcal{F}(\Omega,M)$, that is, $x_0(\cdot)=x(0,\cdot)$, there exists a piecewise constant control signal $u:[0,T]\rightarrow\mathbb{R}^m$ that steers the system into an $\e$-neighborhood of a desired target state $x_F\in\mathcal{F}(\Omega,M)$ in a finite time $T>0$, i.e., $d(x(T,\cdot),x_F(\cdot))<\e$, where $d:\mathcal{F}(\Omega,M)\times \mathcal{F}(\Omega,M)\rightarrow\mathbb{R}$ is a metric on $\mathcal{F}(\Omega,M)$. %Note that the final time $T$ may depend on $\e$.
\end{definition}

In particular, when $\Omega$ is compact, $M$ is a Riemannian manifold, and $\mathcal{F}(\Omega,M)=C(\Omega,M)$ is the space of continuous $M$-valued functions defined on $\Omega$, the metric $d$ in Definition \ref{def:controllability} can be chosen as the uniform metric given by $d(f,g)=\sup_{\b\in\Omega}d_M\big(f(\b),g(\b)\big)$ for any $f,g\in C(\Omega,M)$, where $d_M:M\times M\rightarrow\mathbb{R}$ is the metric induced by the Riemannian metric on $M$. Ensemble controllability defined through the uniform metric is then referred to as \emph{uniform ensemble controllability}.

In this paper, we focus on uniform ensemble controllability of time-invariant linear ensemble systems of the form
\begin{equation}
\label{eq:ensemble_linear}
\frac{d}{dt}x(t,\b)=A(\b)x(t,\b)+B(\b)u(t)
\end{equation}
with the system matrix $A\in C(K,\mathbb{R}^{n\times n})$, control matrix $B\in C(K,\mathbb{R}^{n\times m})$, and state $x(t,\cdot)\in C(K,\mathbb{R}^n)$, where $K$ is a compact subset of $\mathbb{R}$. In this case, the eigenvalues of $A$, denoted by $\l_1$, $\dots$, $\l_n$, are also continuous functions in $\b$. 
%Moreover, because of the compactness of $K$, the uniform norm $\|\cdot\|_\infty$, given by $\|f\|_\infty=\sup_{\b\in K}\|f(\b)\|$ for any $f\in C(K,\mathbb{R}^n)$ and some norm $\|\cdot\|$ on $\mathbb{R}^n$, is well-defined on the state-space $C(K,\mathbb{R}^n)$ such that $C(K,\mathbb{R}^n)$ becomes a Banach space. This norm further induces the uniform metric on $C(K,\mathbb{R}^n)$ as $d(f,g)=\|f-g\|_\infty$, and ensemble controllability defined through this metric is then referred to as \emph{uniform ensemble controllability}.

%======================= Separating points ====================
\subsection{The Technique of Separating Points}
\label{sec:separating_pts}
This section devotes to introducing the technique of separating points developed in our recent work \cite{Li_SICON20}, which provides a systematic way to facilitate uniform ensemble controllability analysis for time-invariant linear ensemble systems in the form of \eqref{eq:ensemble_linear} by using classical controllability of each individual system in these ensembles. This approach is highly nontrivial and probably counterintuitive, because classical controllability of individual systems generally does not imply ensemble controllability of the whole ensemble as shown in the following examples.

%====================== Example 1 =======================
\begin{example}
\label{ex:uncontrollable_1d}
Consider the following single-input scalar linear ensemble system defined on $C([-1,1],\mathbb{R})$
\begin{equation}
\label{eq:uncontrollable_1d}
\frac{d}{dt}x(t,\b)=\b^2x(t,\b)+u(t).
\end{equation}
For each fixed $\b_0\in[-1,1]$, the system indexed by $\b_0$ is controllable on $\mathbb{R}$. However, the reachable set of the whole ensemble, that is, the closure $\overline{\L}$ of the Lie algebra $\L={\rm span}\{1,\b^2,\b^4,\dots\}$ generated by the system matrix $\beta$ and control matrix 1, only contains even functions, which is a proper subset of $C([-1,1],\mathbb{R})$, and hence the system in \eqref{eq:uncontrollable_1d} is not uniformly ensemble controllable \cite{Li_TAC16}.
\end{example}

%====================== Example 2 =======================
\begin{example}
\label{ex:uncontrollable_2d}
Consider the following single-input ensemble system defined on $C([1,2],\mathbb{R}^2)$,
\begin{equation}
\label{eq:uncontrollable_2d}
\frac{d}{dt}x(t,\b)=\b\left[\begin{array}{cc} 1 & 0 \\ 0 & 2 \end{array}\right]x(t,\b)+\left[\begin{array}{c} 1  \\ 1 \end{array}\right]u(t).
\end{equation}
Similarly to the previous example, each individual system in this ensemble indexed by a fixed $\b_0\in[1,2]$ is controllable on $\mathbb{R}^2$, due to the full rank of the controllability matrix
$$\left[\begin{array}{cc} 1 & 1 \\ \b_0 & 2\b_0 \end{array}\right].$$
However, any element $f=\left[\begin{array}{c} f_1  \\ f_2 \end{array}\right]$ in the reachable set $\overline\L$, where 
$$\L={\rm span}\Big\{\left[\begin{array}{c} \b^k  \\ (2\b)^k \end{array}\right]:k\in\mathbb{N}\Big\}$$
is the Lie algebra generated by the system and control matrices,
must satisfy $f_1(2)=f_2(1)$, which implies $\overline\L\subsetneq C([1,2],\mathbb{R}^2)$, and hence the system in \eqref{eq:uncontrollable_2d} is not uniformly ensemble controllable on $C([1,2],\mathbb{R}^2)$.
\end{example}

In fact, Examples \ref{ex:uncontrollable_1d} and \ref{ex:uncontrollable_2d} demonstrate the only two situations that may lead to ensemble uncontrollability of a time-invariant linear ensemble system providing classical controllability of each of its individual system: 
\begin{itemize}
\item Non-injectivity of eigenvalue functions, i.e., some eigenvalue functions of the system matrix have multiple injective branches. As shown in the system in \eqref{eq:uncontrollable_1d}, the eigenvalue function of its systems matrix is $\l(\b)=\b^2$ with $\b\in[-1,1]$, which is not injective, and to be more specific, has two injective branches $[-1,0)$ and $(0,1]$.
\item Shared spectra, i.e., the ranges of some eigenvalue functions of the system matrix have nonempty intersection. As shown in the system in \eqref{eq:uncontrollable_2d}, the eigenvalue functions of its systems matrix are $\l_1(\b)=\b$ and $\l_2(\b)=2\b$ with $\b\in[1,2]$, whose ranges are not disjoint as $\l_1([1,2])\cap\l_2([1,2])=[1,2]\cap[2,4]=\{2\}$.
\end{itemize}

Motivated by the above obervations, the central idea of the separating point technique is to separate points in different injective branches of each eigenvalue function and shared spectra of the system matrix by multiple control inputs. 

%==================== Example 3 ====================
\begin{example}
\label{ex:controllable_1d}
In this example, we revisit the ensemble system in \eqref{eq:uncontrollable_1d}, which is not uniformly ensemble uncontrollable on $C([-1,1],\mathbb{R})$ due to the failure of separating points in the two injective branches $[-1,0)$ and $(0,1]$ of the system matrix $\l(\b)=\b^2$ by the sole control input $u$. To fix this, we apply another control $v$ to the system as
\begin{equation}
\label{eq:controllable_1d}
\frac{d}{dt}x(t,\b)=\b^2x(t,\b)+u(t)+\b v(t).
\end{equation}
Consequently, the Lie algebra generated by the system and control matrices $\L={\rm span}\{1,\b,\b^2,\b^3,\b^4,\dots\}$ contains all the monomials in $\b$ so that $\overline\L=C([-1,1],\mathbb{R})$ by the Weierstrass approximation theorem. Therefore, the system in \eqref{eq:controllable_1d} is uniformly ensemble controllable on $C([-1,1],\mathbb{R})$. 

This example further illuminates that, to guarantee uniform ensemble controllability of a scalar linear ensemble system, it is necessary that the number of control inputs applied to the system is greater than or equal to the number of injective branches of its drift.
\end{example}

%==================== Example 4 ====================
\begin{example}
\label{ex:controllable_2d}
As explained before, the cause of uncontrollability of the ensemble system in \eqref{eq:uncontrollable_2d} is the lack of enough control inputs to separate the point in the shared spectrum of the system matrix, that is, $\l_1([,12])\cap\l_2([1,2])=\{2\}$. Similar to the previous example, we apply another control input $v$ to the system as
\begin{equation}
\label{eq:controllable_2d}
\frac{d}{dt}x(t,\b)=\b\left[\begin{array}{cc} 1 & 0 \\ 0 & 2 \end{array}\right]x(t,\b)+\left[\begin{array}{c} 1  \\ 1 \end{array}\right]u(t)+\left[\begin{array}{c} 0  \\ 1 \end{array}\right]v(t),
\end{equation}
whose reachable set is the uniform closure of 
$$\L={\rm span}\Big\{\left[\begin{array}{c} \b^k  \\ (2\b)^k \end{array}\right],\left[\begin{array}{c} 0 \\ (2\b)^l \end{array}\right]:k,l\in\mathbb{N}\Big\}.$$
For any $f=\left[\begin{array}{c} f_1  \\ f_2 \end{array}\right]\in C([1,2],\mathbb{R}^2)$ and $\e>0$, the Weierstrass approximation theorem guarantees the existence of degree $N$ polynomials $p_1(\b)=\sum_{k=0}^Nc_{1k}\beta^k$ and $p_2=\sum_{k=0}^Nc_{2k}\beta^k$ such that $\|f-p\|_\infty<\e$, where $p=\left[\begin{array}{c} p_1  \\ p_2 \end{array}\right]$. To see $p\in\L$, we rewrite it in the following form
$$p=\sum_{k=1}^Nc_{1k}\left[\begin{array}{c} \b^k  \\ (2\b)^k \end{array}\right]+\sum_{k=1}^N\frac{c_{2k}-2^kc_{1k}}{2^k}\left[\begin{array}{c} 0 \\ (2\b)^k \end{array}\right].$$
This then shows $\overline\L=C([1,2],\mathbb{R}^2)$, and hence concludes uniform ensemble controllability of the system in \eqref{eq:controllable_2d} on $C([1,2],\mathbb{R}^2)$.

Parallel to the observation on the minimum number of control inputs guaranteeing uniform ensemble controllability of scalar linear ensemble systems revealed in  Example \ref{ex:controllable_1d}, this example draws a similar conclusion for multi-dimensional linear ensemble systems, that is, uniform ensemble controllability also requires the number of control inputs to be greater than or equal to the number of shared spectra.  
\end{example}

One of the major contributions of the separating point technique is to examine ensemble controllability through classical controllability, and the tool provided by this technique to bridge the gap between these two types of controllability is reparameterization of ensemble systems by the eigenvalue functions of their system matrices. The main idea can be well illuminated by comparing the reparameterization of the systems in \eqref{eq:uncontrollable_2d} and \eqref{eq:controllable_2d} by their eigenvalue functions $\eta_1=\l_1(\b)$ and $\eta_2=\l_2(\b)$ as 

\begin{equation}
\label{eq:uncontrollable_2d_reparam}
\frac{d}{dt}x(t,\eta)=\left[\begin{array}{cc} \eta_1 & 0 \\ 0 & \eta_2 \end{array}\right]x(t,\eta)+\left[\begin{array}{c} 1  \\ 1 \end{array}\right]u(t)
\end{equation}
and
\begin{equation}
\label{eq:controllable_2d_reparam}
\frac{d}{dt}x(t,\eta)=\left[\begin{array}{cc} \eta_1 & 0 \\ 0 & \eta_2 \end{array}\right]x(t,\eta)+\left[\begin{array}{c} 1  \\ 1 \end{array}\right]u(t)+\left[\begin{array}{c} 0  \\ 1 \end{array}\right]v(t),
\end{equation}
in which the new parameter vector $\eta=(\eta_1,\eta_2)$ takes values on the product of the ranges of $\l_1$ and $\l_2$ as $\l_1([1,2])\times\l_2([1,2])=[1,2]\times[2,4]$. Notice that each individual system in the ensemble \eqref{eq:controllable_2d_reparam} is controllable on $\mathbb{R}^2$, implied by the full rank of the controllability matrix
$$\left[\begin{array}{cccc} 1 & 0 & \eta_1 & 0 \\ 1 & 1 & \eta_2 & \eta_2 \end{array}\right].$$
On the contrary, in the ensemble \eqref{eq:uncontrollable_2d_reparam}, the system indexed by the parameter in the shared spectrum $\eta_1=\eta_2=2$ is not controllable on $\mathbb{R}^2$,  because the controllability matrix for this individual system 
$$\left[\begin{array}{cc} 1 & 2  \\ 1 & 2 \end{array}\right]$$
is of rank $1<2$. The above results, i.e., classical controllability and uncontrollability of individual systems in the reparamterized ensembles in \eqref{eq:controllable_2d_reparam} and \eqref{eq:uncontrollable_2d_reparam}, respectively, are exactly consist with uniform ensemble controllability and uncontrollability of the ensemble systems in \eqref{eq:controllable_2d} and \eqref{eq:uncontrollable_2d}, respectively, in the original parameterization. This conclusion is rigorously formulated in the following proposition.

%========================== Proposition 1 =================================
\begin{proposition}
\label{prop:separating_pts}
	Given a time-invariant linear ensemble system defined on $C(K,\mathbb{R}^n)$ in the form of \eqref{eq:ensemble_linear},
	$$\frac{d}{dt}x(t,\b)=A(\b)x(t,\b)+B(\b)u(t),$$
	where the parameter space $K\subset\mathbb{R}$ is compact, $u:[0,T]\rightarrow\mathbb{R}^m$ is piecewise constant, $A\in C(K,\mathbb{R}^{n\times n})$ is diagonalizable with real eigenvalue functions $\l_1$,$\dots,\l_n\in C(K,\mathbb{R})$, and $B\in C(K,\mathbb{R}^{n\times m})$. Then, the following are equivalent. 
	\begin{enumerate}
	\item The system is uniformly ensemble controllable on $C(K,\mathbb{R}^n)$.
	\item The corresponding diagonalized system 
	\begin{equation}
		\label{eq:diagonal}
		\frac{d}{dt}y(t,\b)=\Lambda(\b)y(t,\b)+\widetilde{B}(\b)u(t),
	\end{equation}
	where $\Lambda(\b)={\rm diag}(\l_1(\b),\dots,\l_n(\b))$, is uniformly ensemble controllable on $C(K,\mathbb{R}^n)$.
	\item The system obtained by parameterizing the system of $y(t,\b)$ in \eqref{eq:diagonal} by the eigenvalue functions $\eta_i=\l_i(\b)$, $\dots$, $\eta_n=\l_n(\b)$, i.e., 
		\begin{align}
	\label{eq:reparameterization}
		\frac{d}{dt}\left[\begin{array}{c} z_1(t,\eta_1) \\ \vdots \\  z_n(t,\eta_n) \end{array}\right]=\left[\begin{array}{c} \eta_1z_1(t,\eta_1) \\ \vdots \\ \eta_nz_n(t,\eta_n)\end{array}\right]+\left[\begin{array}{c} D_1(\eta_1) \\ \vdots \\ D_n(\eta_n) \end{array}\right]u(t)
%	\frac{d}{dt}\left[\begin{array}{c} z_1(t,\eta_1) \\ \vdots \\ z_n(t,\eta_n) \end{array}\right]=\left[\begin{array}{ccc} \Lambda_1(\eta_1) &  &   \\ & \ddots & \\ &   & \Lambda_n(\eta_n) \end{array}\right]\left[\begin{array}{c} z_1(t,\eta_1) \\ \vdots \\ z_n(t,\eta_n) \end{array}\right]\\
%	+\left[\begin{array}{c} D_1(\eta_1) \\ \vdots \\ D_n(\eta_n) \end{array}\right]u(t),
	\end{align}
	\end{enumerate}
	is controllable on $\mathbb{R}^N$ for each $n$-tuple $(\eta_1,\dots,\eta_n)\in K_1\times\cdots\times K_n$, where $K_i=\l_i(K)$, $z_i(t,\eta_i)\in\mathbb{R}^{\kappa_i(\eta_i)}$, $N=\sum_{i=1}^n\kappa_i(\eta_i)$, $\kappa_i(\eta_i)$ denotes the cardinality of the inverse image $\lambda_i^{-1}(\eta_i)=\{\b_{\eta_i}^1,\dots,\b_{\eta_i}^{\kappa_i(\eta_i)}\}$ of $\eta_i\in K_i$ under $\lambda_i$, $I_{\kappa_i(\eta_i)}$ is the $\kappa_i(\eta_i)\times\kappa_i(\eta_i)$ identity matrix, and 
	$$D_i(\eta_i)=\left[\begin{array}{c} \tilde b_i(\b_{\eta_i}^1)  \\ \vdots \\ \tilde b_i(\b_{\eta_i}^{\kappa_i(\eta_i)}) \end{array}\right]\in\mathbb{R}^{\kappa_i(\eta_i)\times m}$$
	is called the \emph{Ensemble Controllability Criterion Matrix} associated with the $i^{\rm th}$ state
	$$\frac{d}{dt}y_i(t,\b)=\l_i(\b)y_i(t,\b)+\tilde{b}_i(\b)u(t)$$
	of the ensemble system in \eqref{eq:diagonal}, and $\tilde{b}_i(\b)$ denotes the $i^{\rm th}$ row of $\widetilde{B}(\beta)$.	
\end{proposition}
\begin{proof}
See \cite{Li_SICON20}.
\end{proof}

It is worth noting that the dimension of the subsystem $z(t,\eta_i)$ of the system in \eqref{eq:reparameterization} is exactly equal to $\kappa(\eta_i)$, the number of the injective branches of the $i^{\rm th}$ eigenvalue function $\l_i$ of the system matrix of the system in \eqref{eq:diagonal}, and hence the system in \eqref{eq:ensemble_linear}. This further indicates that the reparameterization procedure unifies the two notions of separating points by transforming points in different injective branches to those in shared spectra. 

%========================== Remark 1 ==============================
\begin{remark}
\label{rmk:Sobolev}
In addition to the diagonalizable case, Proposition \ref{prop:separating_pts} remains valid for non-Sobolev type linear ensemble systems with non-diagonalizable system matrices, guaranteed by the controllability equivalence between them and their diagonalizable counterparts \cite{}. To be more specific, given such a system in the form of \eqref{eq:ensemble_linear} whose system matrix $A\in C(K,\mathbb{R}^{n\times n})$ has real-valued eigenvalue functions $\l_1,\dots,\l_n\in C(K,\mathbb{R})$, then there exists $P\in C(K,{\rm GL}(n,\mathbb{R}))$ such that $T=P^{-1}AP\in C(K,\mathbb{R}^{n\times n})$ is upper triangular with the diagonal entries $\l_1,\dots,\l_n$, where ${\rm GL}(n,\mathbb{R})$ is the general linear group consisting of $n$-by-$n$ invertible real matrices. Then, its diagonalizable counterpart is defined as the linear ensemble system whose system and control matrices are ${\rm diag}(\l_1,\dots,\l_n)$ and $P^{-1}B$, respectively, and this diagonalized system is uniformly ensemble controllable on $C(K,\mathbb{R}^n)$ if and only if the original system in \eqref{eq:ensemble_linear} is uniformly ensemble controllable on $C(K,\mathbb{R}^n)$. A typical class of non-Sobolev type linear ensemble systems are those with constant, i.e., parameter-independent, control matrices when the system matrices are transformed to the upper triangular form. 
\end{remark}

\section{Rational Canonical Forms for Separating Points}
\label{sec:canonical}

As indicated by Proposition \ref{prop:separating_pts} that uniform ensemble controllability can be examined by classical controllability, the technique of separating points provides a powerful tool for tackling control and analysis tasks for infinite-dimensional systems by utilizing finite-dimensional methods. For example, in classical linear systems theory, transforming a finite-dimensional linear system to its controllability canonical form is an effective way to reveal the properties concerning the dynamics of the system, such as controllability and the characteristic polynomial (of its system matrix). Then, the technique of separating points opens up the possibility for the adoption of classical controllability canonical forms to analyze uniform ensemble controllability and spectra of infinite-dimensional linear ensemble systems, which will be the focus of the remaining of the paper.

%======================== Rational canonical froms ================
\subsection{Separating Point Interpretation of the Rational Canonical Form}
\label{sec:rational}

It is well-known that the system matrix of a classical linear system in the controllability canonical form is the Frobenius companion matrix of its characteristic polynomial, but only those systems that can be controllable by single-inputs admit coordinates under which the system matrices are in the companion form \cite{Brockett70}. To extend the classical controllability canonical form to multi-input linear systems, especially, for the purpose of representing uniform ensemble controllability, it is natural to start from the study of the generalized companion form for matrices, called the rational canonical form or Frobenius normal form. In particular, we will interpret the rational canonical form, especially, for diagonalizable matrices, from the perspective of separating points, which in turn motivates its utilization in establishing the uniform ensemble controllability canonical form.

A matrix in the rational canonical form is block diagonal with each block a companion matrix. To reveal the idea of separating points hidden in the rational canonical form, we consider a diagonalized linear system defined on $\mathbb{R}^n$,
\begin{align}
\label{eq:linear_diag}
\frac{d}{dt}x(t)=\Lambda x(t)+Bu(t),
\end{align}
where $\Lambda={\rm diag}(\l_1,\dots,\l_n)\in\mathbb{R}^{n\times n}$ is the system matrix, and $B\in\mathbb{R}^{n\times m}$ is the control matrix. If the system in \eqref{eq:linear_diag} is controllable and $m=1$, i.e., it is a single-input system, then the controllability matrix $P=\left[\begin{array}{cccc} b\mid \Lambda b\mid\cdots\mid \Lambda^{n-1}b \end{array}\right]\in\mathbb{R}^{n\times n}$ is full rank and transforms the system to the controllability canonical form, that is, the linear system with the system and control matrices 
$$C=P^{-1}AP=\left[\begin{array}{ccccc} 0 & 0 & \cdots & 0 & -c_0 \\ 1 & 0 & \cdots & 0 & -c_1 \\ 0 & 1 & \cdots & 0 & -c_2 \\ \vdots & \vdots & \ddots & \vdots & \vdots \\ 0 & 0 & \cdots & 1 & -c_{n-1}\end{array}\right]$$
and
$$\widetilde B=P^{-1}B=\left[\begin{array}{c} 1 \\ 0 \\ \vdots \\ 0 \\ 0 \end{array}\right],$$
respectively, and in particular, $C$ is the \emph{companion matrix} of the monic polynomial polynomial $c_\Lambda(\l)=\l^n+c_{n-1}\l^{n-1}+\cdots+c_1\l+c_0$, which is exactly the characteristic polynomial of the original system matrix $\Lambda$ \cite{Brockett70}. From the perspective of separating point, in this case, the system matrix $\Lambda$ of the original system in \eqref{eq:linear_diag} cannot have shared spectrum in the sense that $\Lambda$ does not have repeated eigenvalue. Otherwise, say $\l_i=\l_j=\eta$ for some $i\neq j$, the subsystem of the system in  \eqref{eq:linear_diag} consisting of the states $x_i$ and $x_j$
\begin{align*}
\frac{d}{dt}x_i(t)&=\eta x_i(t)+b_iu(t),\\
\frac{d}{dt}x_j(t)&=\eta x_j(t)+b_ju(t),
\end{align*}
where $b_i$ and $b_j$ denote the $i^{\rm th}$ and $j^{\rm th}$ rows of $B$, respectively, is not controllable on $\mathbb{R}^2$, because its controllability matrix 
$$\left[\begin{array}{cc}  b_i & \eta b_i \\ b_j & \eta b_j \end{array}\right]$$
is of rank $1<2$, which then contradicts controllability of the whole system in \eqref{eq:linear_diag} on $\mathbb{R}^n$. Algebraically, this further indicates that a necessary condition for a diagonalizable matrix to admit the companion form is the non-existence of repeated eigenvalues. Similarly, when transforming a diagonalizable matrix with repeated eigenvalues to its rational canonical form, each of the companion block must contain distinct eigenvalues and different companion blocks have to share common eigenvalues. Then, to control a linear system with the system matrix similar to a rational canonical form with more than one companion blocks, the key idea to separate these companion blocks by using control inputs. 

Given a linear system defined on $\mathbb{R}^n$
\begin{align}
\label{eq:linear_classical}
\frac{d}{dt}x(t)=Ax(t)+Bu(t),
\end{align}
to see how the repeated eigenvalues of $A\in\mathbb{R}^{n\times n}$ distribute into different companion blocks in its rational canonical form for the purpose of adopting the separating point technique to separate them, we give an algebraic interpretation of the dynamics of the system. In particular, we consider the state-space $\mathbb{R}^n$ of the system in \eqref{eq:linear_classical} as a module over $\mathbb{R}[\l]$, the ring of polynomials with real coefficients, and translate the dynamics of the system into the action of the intermediate variable $\l$ on $\mathbb{R}^n$ by the linear transformation $A$, i.e., $\l\cdot x=Ax$ for all $x\in\mathbb{R}^n$, which in turn gives a concrete description of the module structure as $p(\l)\cdot x=p(A)x$ for $p(\l)\in\mathbb{R}[\l]$. Because $\mathbb{R}^n$ has finite dimension $n$ as a vector space over $\mathbb{R}$, it is certainly finitely generated as a $\mathbb{R}[\l]$-module as well. However, $\mathbb{R}[\l]$ is an infinite-dimensional vector space over $\mathbb{R}$, and hence so is every free $\mathbb{R}[\l]$-module, which then implies that $\mathbb{R}^n$ must be a torsion $\mathbb{R}[\l]$-module. By the fundamental theorem of finitely generated modules over principal ideal domains, the dynamics of the system in \eqref{eq:linear_classical}, i.e., the action $\l\cdot x=Ax$, decomposes the state-space $\mathbb{R}^n$ into a direct sum of cyclic modules, called the \emph{invariant factor decomposition}, as
\begin{align}
\label{eq:invariant_factor_decompostion}
\mathbb{R}^n=\mathbb{R}[\l]/\<a_1(\l)\>\oplus\cdots\oplus\mathbb{R}[\l]/\<a_k(\l)\>,
\end{align}
where $a_1(\l)$, $\dots$, $a_k(\l)\in\mathbb{R}[\l]$ are monic polynomials of degree at least 1, called the \emph{invariant factors} of $A$, satisfy the divisibility condition $a_{i+1}(\l)\mid a_{i}(\l)$, i.e., $a_{i+1}(\l)$ divides $a_{i}(\l)$, for all $i=1,\dots,k-1$, and $\<a_i(\l)\>=\big\{p(\l)\in\mathbb{R}[\l]:p(\l)=a_i(\l)q(\l)\text{ for some }q(\l)\in\mathbb{R}[\l]\big\}$ denotes the ideal in $\mathbb{R}[\l]$ generated by $a_i(\l)$ \cite{Lang02}. Moreover, each cyclic module $\mathbb{R}[\l]/\<a_i(\l)\>$ in the decomposition in \eqref{eq:invariant_factor_decompostion} is indeed a vector subspace of $\mathbb{R}^n$. To see this, we apply the Euclidean algorithm to divide an arbitrary polynomial $p(\l)\in\mathbb{R}[\l]$ by $a_i(\l)$, which results in $p(\l)=a_i(\l)q(\l)+r(\l)$ for some $q(\l),r(\l)\in\mathbb{R}[\l]$ so that the degree of $r(\l)$ is less than the degree $n_i$ of $a_i(\l)$, or equivalently, $p(\l)=r(\l)$ mod $a_i(\l)$. Because $a_i(\l)q(\l)\in\<a_i(\l)\>$, the quotient ring $\mathbb{R}[\l]/\<a_i(\l)\>$ can be identified with the subring of $\mathbb{R}[\l]$ consisting of polynomials with degree less than $n_i$, which is an $n_i$ dimensional vector space over $\mathbb{R}$.

Recall that the monomial $\l$ acts on $\mathbb{R}^n$ by the linear transformation $A$, together with $a_i(\l)=0$ mod $a_i(\l)$, we obtain $a_i(A)x=0$ for all $x\in\mathbb{R}[\l]/\<a_i(\l)\>$. Then, because $a_i(\l)\mid a_1(\l)$ for all $i=1,\dots,k$, we have $a_1(A)x=0$ for all $x\in\mathbb{R}^n$, which implies $a_1(A)=0$, the trivial linear transformation mapping every element in $\mathbb{R}^n$ to 0. Moreover, the divisibility condition of the invariant factors guarantees that $a_1(\l)$ is the polynomial with the minimal degree annihilating $\mathbb{R}^n$, hence it is also called the \emph{minimal polynomial} of $A$ and also denoted by $m_A(\l)$. On the other hand, the multiplication of the invariant factors $\prod_{i=1}^ka_i(\l)$ is a polynomial of degree $\sum_{i=1}^kn_i=n$, the dimension of $\mathbb{R}^n$, annihilating $\mathbb{R}^n$, and hence it must coincide with the characteristic polynomial $c_A(\l)$ of $A$.

The invariant factor decomposition of $\mathbb{R}^n$ in \eqref{eq:invariant_factor_decompostion} further provides a firm evidence for presenting the linear transformation $x\mapsto\lambda\cdot x=Ax$ as a block diagonal matrix so that each subspace $\mathbb{R}[\l]/\<a_i(\l)\>$ in this decomposition is invariant under the action of $\mathbb{R}[\l]$ on $\mathbb{R}^n$. The remaining task is to find a basis for $\mathbb{R}^n$ consisting of the bases for these subspaces $\mathbb{R}[\l]/\<a_i(\l)\>$, under which the matrix representation of the linear transformation $A$ will be a block diagonal matrix with the blocks representing its restriction to these subspaces. To this end, leveraging the cyclic structure of the $n_i$ dimension subspace $\mathbb{R}[\l]/\<a_i(\l)\>$, we can choose $1$ mod $a_i(\l)$, $\l$ mod $a_i(\l)$, $\dots$, $\l^{n_i-1}$ mod $a_i(\l)$ as its basis, which, in terms of the action of $A$ on $\mathbb{R}^n$, can be represented by 
$b_i$, $Ab_i$, $\dots$, $A^{n_i-1}b_i$ for some $b_i\in\mathbb{R}^n$. Let $a_i(\l)=\l^{n_i}+c_{in_i-1}\l^{n_i-1}+\dots+c_{i1}\l+c_{i0}$, then the action of $A$ restricted to $\mathbb{R}[\l]/\<a_i(\l)\>$ on this basis is given by
\begin{align*}
b_i\ &\mapsto\ Ab_i\\
Ab_i\ &\mapsto\ A^2b_i\\
&\ \vdots\nonumber\\
A^{n_i-2}b_i\ &\mapsto\ A^{n_i-1}b_i\\
A^{n_i-1}b_i\ &\mapsto\ A^{n_i}b_i=-\sum_{j=o}^{n_i-1}c_{ij}A^{j}b_j
\end{align*}
where the last step follows from $a_i(A)b_i=0$ and also guarantees the invariance of the space spanned by this basis under the action of $A$. Consequently, the matrix representation of $A$ restricted to $\mathbb{R}[\l]/\<a_i(\l)\>$ under this basis is the companion matrix of the polynomial $a_i(\l)$ as 
$$C_i=\left[\begin{array}{ccccc} 0 & 0 & \cdots & 0 & -c_{i0} \\ 1 & 0 & \cdots & 0 & -c_{i1} \\ 0 & 1 & \cdots & 0 & -c_{i2} \\ \vdots & \vdots & \ddots & \vdots & \vdots \\ 0 & 0 & \cdots & 1 & -c_{in_i-1}\end{array}\right].$$
Repeating this procedure for all $i=1,\dots,k$, since $n_1+\dots+n_k=n$, we obtain a basis for $\mathbb{R}^n$ in the form of $b_1$, $Ab_1$, $\dots$, $A^{n_1-1}b_1$, $\dots$, $b_k$, $Ab_k$, $\dots$, $A^{n_k-1}b_k$ so that $\mathbb{R}[\l]/\<a_i(\l)\>={\rm span}\{b_i,Ab_i,\dots,A^{n_i-1}b_i\}$ holds for each $i$. The $A$-invariance property of these subspaces $\mathbb{R}[\l]/\<a_i(\l)\>$ then warrantees the block diagonal form of the matrix representation of $A$ under this basis of $\mathbb{R}^n$ as
$$C=\left[\begin{array}{ccc} C_1 & & \\ & \ddots & \\ & & C_k \end{array}\right],$$
namely, the \emph{rational canonical from} of $A$. Computationally, the matrix $P=\big[b_1\mid Ab_1\mid\cdots\mid A^{n_1-1}b_1\mid\dots\mid b_k\mid Ab_k\mid\cdots\mid A^{n_k-1}b_k\big]$, containing the resulting basis as its column vectors, transforms $A$ to its rational canonical form as $C=P^{-1}AP$.

As discussed previously, in the case that $A$ is diagonalizable, each of the companion blocks $C_i$ in its rational canonical form $C$ cannot have repeated eigenvalues, and hence each invariant factor $a_i(\l)$ of $A$ only has simple roots as well. This in turn implies that the repeated eigenvalues of $A$ must be uniformly distributed across the $k$ companion blocks of $C$. Specifically, if $\eta$ is an eigenvalue of $A$ with the algebraic multiplicity $l\leq k$, i.e., $\eta$ is a multiplicity $l$ root of the characteristic polynomial $c_A(\l)=\prod_{i=1}^na_i(\l)$, then the divisibility condition $a_{i+1}(\l)\mid a_i(\l)$ of the invariant factors implies that $\eta$ is the common simple root to $a_1(\l)$, $\dots$, $a_l(\l)$, and correspondingly, $\eta$ is the common multiplicity 1 eigenvalue to the first $l$ companion blocks $C_1$, $\dots$, $C_l$. Especially, the fact of the rational canonical form of $A$ containing $k$ companion blocks indicates that $k$ is the maximal algebraic multiplicity for the eigenvalues of $A$, or in the terminology of separating points, $A$ has $k$ shared spectra. This nature of the rational canonical form, i.e., decoding the the information of shared spectra in terms of companion blocks, immediately stresses its importance in the study of uniform ensemble controllability by using the technique of separating points. To this end, motivated by the equivalence between ensemble and classical controllability revealed in Proposition \ref{prop:separating_pts}, we first turn our attention to the role of the rational canonical form in classical controllability. On the other hand, this investigation will further indicate the broad applicability of the technique of separating points in linear systems theory, including both classical and ensemble linear systems. 

\subsection{Rational Canonical Form and Classical Controllability}
\label{sec:canoncial_single}

In the previous section, we investigated the impact of the dynamics of a linear system on the algebraic structure of its state-space, which, in particular, leads to the invariant factor decomposition of the state-space along with the rational canonical form of the system matrix. The focus now is shifted to the integration of these results with the technique of separating points to study controllability of linear systems. 

To motivate the idea, we still start from a linear system defined on $\mathbb{R}^n$ in the form of \eqref{eq:linear_classical} whose system matrix $A\in\mathbb{R}^{n\times n}$ is similar to a companion matrix. 

%========================= Lemma 1 ============================
\begin{lemma}
\label{lem:single_input}
Given a time-invariant linear system defined on $\mathbb{R}^n$
$$\frac{d}{dt}x(t)=Ax(t)+Bu(t)$$
with $A\in\mathbb{R}^{n\times n}$ similar to a companion matrix and $B\in\mathbb{R}^{n\times m}$. Then, the system is controllable on $\mathbb{R}^n$ if and only if there exists $b\in{\rm Im}(B)$ generating $\mathbb{R}^n$ under the action of $A$, i.e., ${\rm span}\{b,Ab,\dots,A^{n-1b}\}=\mathbb{R}^n$, where ${\rm Im}(B)=\big\{y\in\mathbb{R}^n:y=Bx\text{ for some }x\in\mathbb{R}^m\big\}$ denotes the range space of $B$.
\end{lemma}
\begin{proof}
Regarding a companion matrix as a matrix in the rational canonical form with only one companion block, we conclude that $A$ has only one invariant factor, that is, its characteristic polynomial $c_A(\l)$, which is also equal to its minimal polynomial $m_A(\l)$. Then, the invariant factor decomposition of $\mathbb{R}^n$  under the dynamics of the system, i.e., the action of $A$ on $\mathbb{R}^n$ only has one component as
$$\mathbb{R}^n=\mathbb{R}[\l]/\<c_A(\l)\>.$$
This then indicates that the dynamics of the system induces a global cyclic structure on $\mathbb{R}^n$ so that there exists some $b\in\mathbb{R}^n$ generating $\mathbb{R}^n$, namely, $b$, $Ab$, $\dots$, $A^{n-1}b$ form a basis of $\mathbb{R}^n$. 

Necessity: If $b\in{\rm Im}(B)$, then the space generated by ${\rm Im}(B)$ under the action of $A$ is the whole space $\mathbb{R}^n$, which is exactly equal to the range space ${\rm Im}(W)$ of the controllability matrix $W=\big[\ B\mid AB\mid \cdots\mid A^{n-1}B\ \big]$ of the system. Equivalently,  the rank of $W$ is $n$, which implies controllability of the system on $\mathbb{R}^n$.

Sufficiency: if ${\rm Im}(B)$ does not contain any element generating $\mathbb{R}^n$ under the action of $A$, the minimal polynomial of the linear operator $A|_{{\rm Im}(W)}$, the restriction of $A$ on the space ${\rm Im}(W)\subseteq\mathbb{R}^n$ generated by ${\rm Im}(B)$, must have degree less than the degree $n$ of the minimal polynomial of $A$. Therefore, ${\rm Im}(W)$ is a proper subspace of $\mathbb{R}^n$, which results in uncontrollability of the system.   
\end{proof}

Lemma \ref{lem:single_input} casts a glance at how the interaction between the algebraic structure of the state-space of a linear system, induced by the system dynamics, and the control matrix determines controllability of the system. As a direct consequence, a linear system whose system matrix in the rational canonical form contains a single companion block requires only one control input to be controllable. By thinking about the function of control inputs as separating shared spectra in system matrices as mentioned in Section \ref{sec:separating_pts}, the above observation reinforces the conclusion that no point in the shared spectra of companion matrices needs to be separated to guarantee controllability, and particularly, diagonalizable companion matrices do not have shared spectra as explained at the end of Section \ref{sec:rational}. Going along with this idea, intuitively, if the rational canonical form of the system matrix of a linear system contains more than one companion blocks, then, to guarantee controllability, the number of control inputs applied to the system should be no less than the number of companion blocks. The following proposition then verifies this intuition.

\begin{proposition}
\label{prop:number_of_inputs}
Consider a time-invariant linear system defined on $\mathbb{R}^n$ 
$$\frac{d}{dt}x(t)=Ax(t)+Bu(t),$$
where the rational canonical form of the system matrix $A\in\mathbb{R}^{n\times n}$ consists of $k$ companion blocks. If the system is controllable on $\mathbb{R}^n$, then the control matrix $B\in\mathbb{R}^{n\times m}$ satisfies $m\geq k$.
\end{proposition}
\begin{proof}
We prove this proposition by contradiction and assume controllability of the system on $\mathbb{R}^n$ with $m<k$. 

By the definition of the invariant factor decomposition, the $k$ companion blocks in the rational canonical form of the system matrix $A$ corresponds to $k$ invariant factors $a_1(\l)$, $\dots$, $a_k(\l)$ of $A$, so that $\mathbb{R}^n$ can be decomposed into a direct sum of $k$ cyclic subspaces under the system dynamics as
$$\mathbb{R}^n=\mathbb{R}[\l]/\<a_1(\l)\>\oplus\cdots\oplus\mathbb{R}[\l]/\<a_k(\l)\>,$$
and each $\mathbb{R}[\l]/\<a_1(\l)\>$ is $A$-invariant. On the other hand, controllability of the system indicates that $\mathbb{R}^n$ is generated by ${\rm Im}(B)$ under the action of $A$ as  
\begin{align*}
\mathbb{R}^n&={\rm span}\{A^ib_j:i=0,\dots,{n-1},j=1,\dots,m\}\nonumber\\
&={\rm Im}(B)+{\rm Im}(AB)+\dots+{\rm Im}(A^{n-1}B)
\end{align*}
where $b_j$ denotes the $j^{\rm th}$ row of $B$ for each $j=1,\dots,m$ and ``$+$" denotes the sum of vector spaces, i.e., every element $w\in\mathbb{R}^n$ can be represented as $w=\sum_{i=0}^{n-1}w_i$ for some $w_i\in{\rm Im}(A^iB)$, $i=0,\dots,n-1$, and this also gives a cyclic structure of $\mathbb{R}^n$. Then, the uniqueness of the invariant factor decomposition yields $\mathbb{R}[\l]/\<a_i(\l)\>=P_iW_i$ for cyclic subspace $W_{i}={\rm span}\{v_{i},Av_{i},\dots, A^{n-1}v_{i}\}$ generated by $v_i\in{\rm Im}(B)$ and all $i=1,\dots,n$, where where $P_i:\mathbb{R}^n\rightarrow\mathbb{R}[\l]/\<a_i(\l)\>$ denotes the projection operator onto $\mathbb{R}[\l]/\<a_i(\l)\>$. The $A$-invariance of $\mathbb{R}[\l]/\<a_i(\l)\>$ further gives the commutativity $AP_i=P_iA$ so that
\begin{align}
\label{eq:cyclic_subspace}
P_iW_i&={\rm span}\{P_iv_{i},P_iAv_{i},\dots, P_iA^{n-1}v_{i}\}\nonumber\\
&={\rm span}\{v_{i},AP_iv_{i},\dots, A^{n-1}P_iv_{i}\}\nonumber\\
&={\rm span}\{v_{i},AP_iv_{i},\dots, A^{n_i-1}P_iv_{i}\},
\end{align}
where $n_i$ denotes the degree of $a_i(\l)$ and the last step follows from $P_iv_i\in\mathbb{R}[\l]/\<a_i(\l)\>$ and $a_i(A)v=0$ restricted to $\mathbb{R}[\l]/\<a_i(\l)\>$.

However, because $m<k$, the generators $v_1,\dots,v_k\in{\rm Im}(B)$ of the cyclic subspaces of $\mathbb{R}^n$ must be linearly dependent. Without loss of generality, we assume that $v_i=\alpha v_j=v$ for some $\alpha\in\mathbb{R}$ and $i<j$, then we have
\begin{align}
&\mathbb{R}[\l]/\<a_i(\l)\>\oplus\mathbb{R}[\l]/\<a_j(\l)\>=P_iW_i\oplus P_jW_j=(P_i\oplus P_j)\mathbb{R}^n\nonumber\\
&={\rm span}\{(P_i\oplus P_j)A^{\alpha}v_\beta:\alpha=0,\dots,n-1,\beta=1,\dots,k\}\nonumber\\
&={\rm span}\{A^{\alpha}(P_i\oplus P_j)v_\beta:\alpha=0,\dots,n-1,\beta=1,\dots,k\}\nonumber\\
&={\rm span}\{(P_i\oplus P_j) v,A(P_i\oplus P_j)v,\dots,A^{n-1}(P_i\oplus P_j)v\} \label{eq:controadiction}
\end{align}
Let $n_i$ and $n_j$ be the degree of the invariant factors $a_i(\l)$ and $a_j(\l)$, respectively, then the dimension of the space $\mathbb{R}[\l]/\<a_i(\l)\>\oplus\mathbb{R}[\l]/\<a_j(\l)\>$ is $n_i+n_j$. By the assumption $i<j$, we have $a_j(\l)\mid a_i(\l)$, which implies
\begin{align*}
&a_i(A)(P_i\oplus P_j) v=a_i(A)(P_iv+P_jv)\\
&=a_i(A)P_iv+\frac{a_i(\l)}{a_j(\l)}a_j(A)P_jv=0.
\end{align*}
Therefore, the dimension of ${\rm span}\{(P_i\oplus P_j) v,A(P_i\oplus P_j)v,\dots,A^{n-1}(P_i\oplus P_j)v\}$ is at most $n_i$, which contradicts the equality in \eqref{eq:controadiction}.
\end{proof}

The condition $m\geq k$ in Proposition \ref{prop:number_of_inputs} can also be interpreted in the language of separating points as the shared spectra of the system matrix distributed in different companion blocks in its rational canonical form requires different control inputs to separate. However, technically, the proof of Proposition \ref{prop:number_of_inputs}, especially the contradiction of the equality in \eqref{eq:controadiction}, also indicates that this condition is not sufficient, since controllability also requires the control matrix to offer linear independent generators to all the cyclic subspaces in the invariant factor decomposition of the state-space induced by the system dynamics, which correspond to the companion blocks in the rational canonical form of the system matrix in the one-to-one fashion. We now summarize this observation as a corollary of Proposition \ref{prop:number_of_inputs}.

%Recall that different companion blocks in the rational canonical form of a matrix contain shared spectra, i.e., repeated eigenvalues of the matrix. This observation gives rise to a conceptual interpretation of the condition $m\geq k$ in Proposition \ref{prop:number_of_inputs} in terms of separating points as shared spectra distributed in different companion blocks in the rational canonical form of the system matrix of a linear system require different control inputs to separate for guaranteeing controllability of the system. Equivalently, one control input is only possible to arbitrarily steer the system on one component of the invariant factor decomposition of the state space under the action of the system matrix.

%In addition, technically, the proof of Proposition \ref{prop:number_of_inputs} also shows that the necessary controllability condition $m\geq k$ is not sufficient. To be more specific, controllability of the system in \eqref{eq:system_multi_input} also requires linear independence of the generators of the cyclic subspaces in the invariant factor decomposition. In particular, we summarize this observation as a corollary of Proposition \ref{prop:number_of_inputs}.

%========================= Corollary 1 ============================
\begin{corollary}
\label{cor:cyclic_structure}
Let $a_1(\l)$, $\dots$, $a_k(\l)$ be the invariant factors of the system matrix $A$ of the linear system $\frac{d}{dt}x(t)=Ax(t)+Bu(t)$ defined on $\mathbb{R}^n$. Then, the system is controllable on $\mathbb{R}^n$ if and only if ${\rm Im}(B)$ has $k$ basis elements $v_1$, $\dots$, $v_k$ such that 
$$\mathbb{R}[\l]/\<a_i(\l)\>={\rm span}\{P_iv_i,AP_iv_i,\dots,A^{n_i-1}P_ib_i\}$$
with $P_i:\mathbb{R}^n\rightarrow\mathbb{R}[\l]/\<a_i(\l)\>$ and $n_i$ denoting the projection operators onto $\mathbb{R}[\l]/\<a_i(\l)\>$ and the degree of $a_i(\l)$, respectively, for all $i=1,\dots,k$.
\end{corollary}
\begin{proof}
The necessity has been shown in the proof of Proposition \ref{prop:number_of_inputs} by contradiction, meaning, if ${\rm Im}(B)$ does not contain such basis elements, then the system cannot be controllable. Therefore, we focus on the sufficiency. 

Let $W_i={\rm span}\{v_i,Av_i,\dots,A^{n-1}v_i\}$ denote the cyclic subspace of $\mathbb{R}^n$ generated by $v_i$ under the system dynamics, then we have $\mathbb{R}[\l]/\<a_i(\l)\>=P_iW_i\subseteq W_i$, as shown in the proof of Proposition  \ref{prop:number_of_inputs}. On the other hand, the linear independence of $v_1,\dots,v_k$ implies 
\begin{align}
\label{eq:invariant_factor_system}
P_1W_1\oplus\cdots\oplus P_kW_k\subseteq W_1+\cdots+W_k.
\end{align}
Note that the left hand side of \eqref{eq:invariant_factor_system} gives the invariant factor decomposition of $\mathbb{R}^n$ under the the action of $A$, and hence $W_1+\cdots+W_k=\mathbb{R}^n$ holds as well. In addition, it is not hard to see $W_1+\cdots+W_k={\rm Im}(W)$, the range space of the controllability matrix $W=\big[\ B\mid AB\mid\cdots\mid A^{n-1}B\big]$ of the system, which then yields ${\rm rank}(W)=n$ indicating controllability of the system. 
\end{proof}

The main contribution of Corollary \ref{cor:cyclic_structure} is that it provides a necessary and sufficient controllability condition for multi-input linear systems. Meanwhile, it also generalizes Proposition \ref{prop:number_of_inputs} by conveying the intuitive idea: in addition to the number of control inputs, the way in which these inputs involve in the system, i.e., the structure of the control matrix, also plays a crucial role to guarantee controllability of the system.

Due to various favorable properties of the rational canonical form of a matrix, e.g., algebraically characterizing the shared spectra and the invariant subspaces of the action of the matrix, it is preferable to represent the dynamics of linear systems, i.e., the system matrices, in the rational canonical form. From the perspective of linear algebra, this amounts to the search of an appropriate basis for the state-space of the system that is consistent with its invariant factor decomposition under the system dynamics. To this end, we notice that the condition $\mathbb{R}[\l]/\<a_i(\l)\>={\rm span}\{P_iv_i,AP_iv_i,\dots,A^{n_i-1}P_ib_i\}$ with $v_i\in{\rm Im}(B)$ in Corollary \ref{cor:cyclic_structure} gives a strong hint of constructing such a basis by using the control matrix $B$.

%========================= Theorem 1 ============================
\begin{theorem}[Classical controllability canonical form]
\label{thm:canonical_single}
Let $\frac{d}{dt}x(t)=Ax(t)+Bu(t)$ with $A\in\mathbb{R}^{n\times n}$ and $B\in\mathbb{R}^{n\times m}$ be a controllable system on $\mathbb{R}^n$. Then, there exists a basis for $\mathbb{R}^n$ such that the representation of the system in the coordinates with respect to this basis has the form
\begin{align}
\label{eq:canonical_single}
\frac{d}{dt}y(t)=\tilde Ay(t)+\tilde Bu(t),
\end{align}
where  
\begin{align*}
C=\left[\begin{array}{ccc} C_1 & & \\ & \ddots & \\ & & C_k \end{array}\right]
\end{align*}
 is the rational canonical form of $A$, and 
 \begin{align*}
\widetilde B=\left[\begin{array}{ccccc} \tilde b_{11} & \cdots & \tilde b_{1k} & \cdots & \tilde b_{1l} \\  & \ddots & \vdots & & \vdots \\  & & \tilde b_{kk} & \cdots & \tilde b_{kl} \end{array}\right]
\end{align*}
is a block upper triangular matrix. Equivalently, there is $P\in{\rm GL}(n,\mathbb{R})$ such that $y(t)=Px(t)$, $C=P^{-1}AP$ and $\widetilde B=P^{-1}B$, where ${\rm GL}(n,\mathbb{R})$ is the group of invertible $n$-by-$n$ real matrices. 
\end{theorem}
\begin{proof}
Let $a_1(\l)$, $\dots$, $a_k(\l)$ be the invariant factors of $A$ of the degree $n_1$, $\dots$, $n_k$, respectively, and $\mathbb{R}^n=\mathbb{R}[\l]/\<a_1(\l)\>\oplus\cdots\oplus\mathbb{R}[\l]/\<a_k(\l)\>$ be the invariant factor decomposition of $\mathbb{R}^n$ under the action of $A$. The proof then follows from inductively constructing a basis of $\mathbb{R}^n$, consisting of the bases of the $A$-invariant subspaces $\mathbb{R}[\l]/\<a_i(\l)\>$, by using the column vectors $b_1$, $\dots$, $b_m$ of $B$. 

We first assume $n_1>\cdots>n_k$, and start from constructing a basis for $\mathbb{R}[\l]/\<a_1(\l)\>$. To this end, we pick $v_1=w_1\in{\rm Im}(B)$ such that $a_2(A)v_1\neq0$. The existence of $v_1$ can be shown by contradiction: otherwise, no element in ${\rm Im}(B)$ can generate an $n_1$-dimensional subspace of $\mathbb{R}^n$ under the action of $A$, and hence $\mathbb{R}[\l]/\<a_1(\l)\>$, which contradicts controllability of the system by Corollary \ref{cor:cyclic_structure}. Moreover, because $a_i(\l)\mid a_2(\l)$ for all $i=3,\dots,k$, $a_2(A)v_1\neq0$ guarantees $a_i(A)v_1\neq0$ for all $i=3,\dots,k$. Consequently, the cyclic subspace of $\mathbb{R}^n$ generated by $v_1$ under the action of $A$ must be $n_1$-dimensional. Together with $a_1(A)v_1=0$ guaranteeing the $A$-invariance of this subspace ${\rm span}\{v_1,Av_1\dots,A^{n_1-1}v_1\}$, it necessarily coincides with $\mathbb{R}[\l]/\<a_1(\l)\>$. Next, a similarly procedure can be employed to construct a basis for $\mathbb{R}[\l]/\<a_2(\l)\>$. In particular, we pick $w_2\in{\rm Im}\big((I-P_1)B\big)$ such that $v_2=(I-P_1)w_2$, i.e., $v_2$ is in the complement of $\mathbb{R}[\l]/\<a_1(\l)\>$, satisfies $a_2(A)v_2=0$ but $a_3(A)v_2\neq0$, and hence we obtain $\mathbb{R}[\l]/\<a_2(\l)\>={\rm span}\{v_2,Av_2\dots,A^{n_2-1}v_2\}$. Inductively, we can find $w_i\in{\rm Im}(B)$ such that $v_i=(I-P_{i-1})\cdots(I-P_1)w_i$ satisfies $\mathbb{R}[\l]/\<a_i(\l)\>={\rm span}\{v_i,Av_i\dots,A^{n_i-1}v_i\}$ for all $i=1,\dots,k$.

However, if $n_i=n_{i+1}$ for some $i=1,\dots,k-1$, then at the $i^{\rm th}$ iteration of the above construction, we can find $w_i,w_{i+1}\in{\rm Im}(B)$ such that $v_i=(I-P_{i-1})\cdots(I-P_1)w_i$ and $v_{i+1}=(I-P_{i-1})\cdots(I-P_1)w_{i+1}$ are linearly independent and satisfy $\mathbb{R}[\l]/\<a_i(\l)\>={\rm span}\{v_i,Av_i\dots,A^{n_i-1}v_i\}$ and $\mathbb{R}[\l]/\<a_{i+1}(\l)\>={\rm span}\{v_{i+1},Av_{i+1}\dots,A^{n_{i+1}-1}v_i\}$.

Because $n=n_1+\dots+n_k$, the collection of the bases $v_1$, $\dots$, $A^{n_1-1}v_1$, $\dots$, $v_k$, $\dots$, $A^{n_k-1}v_k$ for all the cyclic subspaces of $\mathbb{R}^n$ invariant under the system dynamics forms a basis for the whole $\mathbb{R}^n$, under which the matrix representation $C$ of the linear transform $x\mapsto Ax$ is in the rational canonical form by the definition of the invariant factor decomposition and rational canonical form in Section \ref{sec:rational}. 

To see that $B$ admits an upper triangular matrix representation under this basis, we first note that because the invariant factor decomposition is a direct sum decomposition, $\mathbb{R}[\l]/\<a_i(\l)\>\cap\mathbb{R}[\l]/\<a_j(\l)\>=\{0\}$ holds for all $i\neq j$. As a result, there is a partition $\{I_1,\dots,I_{k+1}\}$ on the columns of $B$, namely, $I_i\cap I_j=\varnothing$ for all $i\neq j$ and $\bigcup_{i=1}^{k+1}I_i=\{1,\dots,m\}$, such that $w_i={\rm span}\{b_j:j\in I_i\}$, i.e., the generators of different $A$-invariant cyclic subspaces of $\mathbb{R}^n$ must be generated by different column vectors of $B$. Without loss of generality, we can order the column vectors of $B$ in the way that $\max I_i<\min I_j$ for $i<j$, i.e., every element in $I_i$ is less than that in $I_j$ if $i<j$. Then, the construction of the basis as $v_i\in{\rm Im}\big((I-P_{i-1})\cdots(I-P_1)B\big)$ for all $i=1,\dots,k$ guarantees that the columns $b_j$ of $B$ with $j\in I_1\cup\cdots\cup I_{i-1}$ are complementary to all the $v_l$ with $l\geq i$, which then gives the upper triangular matrix representation $\widetilde B$ of $B$ under this basis. 

Equivalently, in terms of matrices, applying the matrix $P=[\ v_1\mid\cdots\mid A^{n_1-1}v_1\mid\cdots\mid v_k\mid\cdots\mid A^{n_k-1}v_k\ ]\in{\rm GL}(n,\mathbb{R})$, whose column vectors are the basis of $\mathbb{R}^n$ constructed above, as a coordinate transformation to the system as $y(t)=Px(t)$ yields the representation of the system and control matrices as $C=P^{-1}AP$ and $\widetilde B=P^{-1}B$, respectively, in the $y$-coordinate as desired.
\end{proof}

The classical controllability canonical form in \eqref{eq:canonical_single} estabilished for multi-input controllable linear systems is equivalently to the one suggested in \cite{Wonham85} up to a change of coordinates.

%%%%%%%%%%%%%%%%%%%%%%%%%%%%%%%%%%%%%%%%%%%%
\section{Ensemble Controllability Canonical Form for Separating Points}
\label{sec:canoncial_ensemble}

The pivotal role of the separating point technique in ensemble control theory has been witnessed in Proposition \ref{prop:separating_pts}, which gives rise to the equivalence between ensemble and classical controllability. In the previous section, we further strengthen the power of this technique by showing its capability to characterize classical controllability, which leads to the establishment of a classical controllability canonical form for multiple input linear systems. This section then devotes to unify these two notions of the separating points, and particularly, to establish a controllability canonical form for linear ensemble systems by leveraging that for classical linear systems. 

The major challenge to unifying the ideas of separating points for classical and ensemble linear systems is to coordinate different representations of shared spectra, those are, repeated eigenvalues (point spectra) for constant matrices and overlapping images of eigenvalue functions (continuous spectra) for matrix-valued functions, respectively. To overcome this, the essential tool is the reparameterization technique used in Proposition \ref{prop:separating_pts}, that is, the reparameterization of linear ensemble systems by the eigenvalue functions of their system matrices, which converts overlapping images of the eigenvalue functions to repeated eigenvalues. This in turn provides the rational canonical form a chance to characterize shared spectra of matrix-valued functions, and hence introduce it to the study of uniform ensemble controllability of linear ensemble systems. 

%==============================================================
\subsection{Canonical Form for Single-Input Linear Ensemble Systems}

To illuminate the main idea, we start from integrating the techniques of separating overlapping images of eigenvalue functions and repeated eigenvalues for single-input linear ensemble systems.

%==================== Lemma 2 ===============================
\begin{lemma}
\label{lem:single_input_ensemble}
Given a non-Sobolev type time-invariant single-input linear ensemble system in the form of
\begin{align}
\label{eq:single_input_ensemble}
\frac{d}{dt}x(t,\beta)=A(\beta)x(t,\beta)+b(\beta)u(t),
\end{align} 
where $\b$ takes values on a compact subset $K$ of $\mathbb{R}$, and $A\in C(K,\mathbb{R}^{n\times n})$ has real-valued eigenvalue functions $\l_1,\dots,\l_n\in C(K,\mathbb{R})$, and $b\in C(K,\mathbb{R}^{n})$. If the system is uniformly ensemble controllable on $C(K,\mathbb{R}^n)$, then 
\begin{enumerate}
\item $\l_i$ is injective for every $i=1,\dots,n$,
\item $\l_i(K)\cap\l_j(K)=\varnothing$ for any $i\neq j$, where $\l_i(K)=\{\l_i(\b)\in\mathbb{R}:\beta\in K\}$ and $\l_j(K)=\{\l_j(\b)\in\mathbb{R}:\beta\in K\}$ denote the images of the functions $\l_i$ and $\l_j$, respectively.
\end{enumerate}
\end{lemma} 
\begin{proof}
Because all the eigenvalue functions of $A$ are real-valued, it is possible to transform the system in \eqref{eq:single_input_ensemble} to the form
\begin{align}
\label{eq:single_input_ensemble_triangular}
\frac{d}{dt}y(t,\beta)=T(\beta)y(t,\beta)+\tilde b(\beta)u(t)
\end{align} 
such that $T\in C(K,\mathbb{R}^{n\times n})$ an upper triangular matrix whose diagonal entries are the eigenvalue functions $\l_k$, $k=1,\dots,n$ of $A$. Next, to apply the technique of separating points as shown in Proposition \ref{prop:separating_pts}, we reparameterize the diagonalized counterpart of the system in \eqref{eq:single_input_ensemble_triangular}, that is,
\begin{align}
\label{eq:single_input_ensemble_diag}
\frac{d}{dt}y(t,\beta)=\Lambda(\beta)y(t,\beta)+\tilde b(\beta)u(t)
\end{align}
with $\Lambda={\rm diag}(\l_1,\dots, \l_n )\in C(K,\mathbb{R}^{n\times n})$, which has the same uniform ensemble controllability property as the system in \eqref{eq:single_input_ensemble} according to Remark \ref{rmk:Sobolev}, by the eigenvalue functions $\eta_k=\l_k(\b)$, $k=1,\dots,n$, and this procedure results in a linear ensemble system parameterized by $\eta=(\eta_1,\dots,\eta_n)\in\l_1(K)\times\cdots\times\l_n(K)$ as
\begin{align}
\label{eq:single_input_ensemble_reparam}
\frac{d}{dt}\left[\begin{array}{c} z_1(t,\eta_1)  \\  \vdots \\ z_n(t,\eta_n) \end{array}\right]=\left[\begin{array}{c} \eta_1z_1(t,\eta_1)  \\  \vdots \\  \eta_nz_n(t,\eta_n) \end{array}\right]+\left[\begin{array}{c} D_1(\eta_1)  \\  \vdots \\ D_n(\eta_n) \end{array}\right]u(t),
\end{align} 
where $D(\eta_i)$ is the Ensemble Controllability Criterion Matrix of the $i^{\rm th}$ state $\frac{d}{dt}y_i(t,\b)=\l_i(\b)y_i(t)+\tilde b_i(\b)u(t)$ of the system in \eqref{eq:single_input_ensemble_diag}  with $\tilde b_i\in C(K,\mathbb{R})$ denoting the $i^{\rm th}$ row of $\tilde b\in C(K,\mathbb{R}^n)$. 

According to Proposition \ref{prop:separating_pts}, uniform ensemble controllability of the system in \eqref{eq:single_input_ensemble} implies classical controllability of each individual system in the ensemble in \eqref{eq:single_input_ensemble_reparam}. Because the system in \eqref{eq:single_input_ensemble_reparam} is single-input, by Proposition \ref{prop:number_of_inputs}, classical controllability requires that the rational canonical form of its system matrix contains only one companion block for each $\eta\in\l_1(K)\times\cdots\times\l_n(K)$. Together with the diagonal form of the system matrix, it cannot have any repeated eigenvalue for each $\eta$. Therefore, $D(\eta_i)\in\mathbb{R}$ and $\eta_i\neq\eta_j$ for all $i\neq j$ must hold for every $\eta$, which, in terms of the $\beta$ parameterization, are exactly the injectivity and disjoint images of all the eigenvalue functions $\l_i$ of $A$, respectively.
\end{proof}

The proof of Lemma \ref{lem:single_input_ensemble} also gives a glimpse at the utilization of the rational canonical form in the study of ensemble controllability. To enhance the role of the rational canonical form in this research direction, we follow the guidance of the intuition that ensemble controllability  of an ensemble system implies classical controllability of each individual system in this ensemble, which can also been seen as follows: because the individual systems in the ensemble in \eqref{eq:single_input_ensemble} are also included in the reparameterized ensemble in \eqref{eq:single_input_ensemble_reparam}, classical controllability of the individuals in \eqref{eq:single_input_ensemble_reparam}, guaranteeing uniform ensemble controllability of the system in \eqref{eq:single_input_ensemble}, indicates classical controllability of the individuals in \eqref{eq:single_input_ensemble}. Then, each individual system, as a controllable single-input classical linear system, in the ensemble in \eqref{eq:single_input_ensemble}, can be transformed to the classical controllability canonical form, in which the system matrix is a companion matrix and the control matrix is the first standard basis vector, i.e., the first entry is 1 and 0 elsewhere, as discussed in Section \ref{sec:rational}. This immediately opens up the possibility of globally transforming a single-input uniformly ensemble controllable linear system to the form in which the system matrix is a ``companion matrix-valued function" and the control matrix is the constant function equal to the first standard basis vector.

%========================== Theorem 2 ==============================
\begin{theorem}[Canonical form for single-input systems]
\label{thm:single_input_ensemble}
Given a single-input uniformly ensemble controllable linear ensemble system defined on $C(K,\mathbb{R}^n)$ as in \eqref{eq:single_input_ensemble}, there exists $P\in C\big(K,{\rm GL}(n,\mathbb{R})\big)$ such that in the coordinate $y(t,\cdot)=P^{-1}x(t,\cdot)$ the system dynamics is governed by  
\begin{align}
\label{eq:single_input_ensemble_canonical}
\frac{d}{dt}y(t,\b)=C(\b)y(t,\b)+\bar{b}(\b)u(t),
\end{align}
where 
\begin{align*}
C(\b)&=P^{-1}(\b)A(\b)P(\b)\\
&=\left[\begin{array}{ccccc} 0 & 0 & \cdots & 0 & -c_0(\b) \\ 1 & 0 & \cdots & 0 & -c_1(\b) \\ 0 & 1 & \cdots & 0 & -c_2(\b) \\ \vdots & \vdots & \ddots & \vdots & \vdots \\ 0 & 0 & \cdots & 1 & -c_{n-1}(\b)\end{array}\right]\in C(K,\mathbb{R}^{n\times n})
\end{align*}
and
$$
\bar{b}(\b)=P^{-1}(\b)b(\b)=\left[\begin{array}{c} 1 \\ 0 \\ \vdots \\ 0 \\ 0 \end{array}\right]\in C(K,\mathbb{R}^n).
$$
\end{theorem}
\begin{proof}
Guaranteed by uniform ensemble controllability of the whole ensemble in \eqref{eq:single_input_ensemble}, all of its individual systems are controllable on $\mathbb{R}^n$. In particular, for the individual system indexed by $\b\in K$, its controllability matrix
$$P(\b)=\left[\begin{array}{cccc} b(\b) & A(\b)b(\b) & \cdots & A^{n-1}(\b)b(\b) \end{array}\right]\in{\rm GL}(n,\mathbb{R})$$
serves as a change of coordinates $y(t,\b)=P^{-1}(\b)x(t,\b)$ giving $C(\b)=P^{-1}(\b)A(\b)P(\b)$ and $\bar{b}(\b)=P^{-1}(\b)b(\b)$.

By now, we have shown the existence of a pointwise transformation of the ensemble system in \eqref{eq:single_input_ensemble} to the form in \eqref{eq:single_input_ensemble_canonical}. The remaining task is to prove the continuity of $P$ and $P^{-1}$, as functions defined on $K$, with respect to $\b\in K$, which results in $C\in C(K,\mathbb{R}^{n\times n})$ and $\bar{b}\in C(K,\mathbb{R}^n)$ so that the ensemble system in \eqref{eq:single_input_ensemble_canonical} indeed evolves on $C(K,\mathbb{R}^n)$. To this end, we note that every entry of $P$ is a polynomial, which is a continuous function, in the entries of $A$ and $b$, together with the continuity of $A$ and $b$, $P$ is also a continuous function as a composition of continuous functions. Furthermore, the continuity of $P^{-1}$ follows from the application of the inverse function theorem to $P^{-1}P=I$, the $n$-by-$n$ identity matrix. 
\end{proof}

It is well known that controllability of linear systems is invariant under change of coordinates, and hence, in Theorem \ref{thm:single_input_ensemble}, the assumption of uniform ensemble controllability of the system in \eqref{eq:single_input_ensemble} implies that of its canonical form in \eqref{eq:single_input_ensemble_canonical}. However, fundamentally different from classical linear systems, the feasibility of transforming a linear ensemble system to the canonical form does not guarantee its uniform ensemble controllability, as shown in the following example. 

%=========================== Example 5 ===============================
\begin{example}
\label{ex:canonical_uncontrollable_2d}
We revisit the ensemble system in \eqref{eq:uncontrollable_2d} in Example \ref{ex:uncontrollable_2d} defined on $C([1,2],\mathbb{R}^2)$,
$$\frac{d}{dt}x(t,\b)=\b\left[\begin{array}{cc} 1 & 0 \\ 0 & 2 \end{array}\right]x(t,\b)+\left[\begin{array}{c} 1  \\ 1 \end{array}\right]u(t).$$
Although the system is not uniformly ensemble controllable on $C([1,2],\mathbb{R}^2)$, it is straightforward to check that the $\mathbb{R}^{2\times 2}$-valued function
$$P(\b)=\left[\begin{array}{cc} 1 & 1 \\ \b & 2\b\end{array}\right]\in C([1,2],\mathbb{R}^{2\times2})$$
transforms it to the canonical form 
\begin{align}
\label{eq:canonical_uncontrollable_2d}
\frac{d}{dt}y(t,\b)=\left[\begin{array}{cc} 0 & -2\b^2 \\ 1 & 3\b \end{array}\right]y(t,\b)+\left[\begin{array}{c} 1  \\  0 \end{array}\right]u(t).
\end{align}
In this case, it is obvious that the system in \eqref{eq:canonical_uncontrollable_2d} is not uniformly ensemble controllable on $C([1,2],\mathbb{R}^2)$, either.
\end{example}

The proof of Theorem \ref{thm:single_input_ensemble} actually provides a clue to this unwanted phenomenon described in Example \ref{ex:canonical_uncontrollable_2d}: the transformation of a simple-input linear ensemble system to its canonical form, equivalently, the continuity and invertibility of the matrix-valued function $P$ constructed pointwisely in the proof of Theorem \ref{thm:single_input_ensemble}, only requires classical controllability of each individual system in the ensemble. Correspondingly, this also gives the converse to Theorem \ref{thm:single_input_ensemble}, that is, a linear ensemble system in the form of \eqref{eq:single_input_ensemble} can be transformed to its canonical form in \eqref{eq:single_input_ensemble_canonical} providing classical controllability of each individual system in the ensemble. However, in Section \ref{sec:separating_pts}, we presented a couple of examples, including the one revisited in Example \ref{ex:canonical_uncontrollable_2d}, to illustrate the non-sufficiency of classical controllability of individual systems for ensemble controllability of linear ensemble systems. 

In addition, it is worth to comment on the case of non-Sobolev type linear ensemble systems. To this end, without loss of generality, we can restrict our attention to systems with diagonalizable system matrices. Then, the necessary and sufficient condition for the existence of such canonical form transformations is that the system matrices do not have any repeated eigenvalues for all the values of the system parameters. For example, in the above Example \ref{ex:canonical_uncontrollable_2d}, evaluated at any $\b\in[1,2]$, the system matrix of the ensemble in \eqref{eq:uncontrollable_2d} has distinct eigenvalues $\l_1(\b)=\b$ and $\l_2(\b)=2\b$, enabling the transformation of the ensemble to its canonical form in \eqref{eq:canonical_uncontrollable_2d}. The failure of its uniform ensemble controllability actually arises from the overlapping image of the eigenvalues of the system matrix as functions in $\b$, that is, $\l_1([1,2])\cap\l_2([1,2])=[1,2]\cap[2,4]=\{2\}\neq\varnothing$, which then reiterates the major distinction of shared spectra between ensemble and classical linear systems.

%=======================================================
\subsection{Functional Rational Canonical Form for Separating Shared Spectra}
\label{sec:canonical_ensemble}

The most effective tool to reduce the disparity between ensemble and classical controllability, along with bridging the gap between ensemble and classical linear systems in terms of separating shared spectra, is the ``spectral reparameterization" procedure used in Proposition \ref{prop:separating_pts}. Indeed, for any linear ensemble system, the feasibility of representing each individual system in the reparamterized ensemble in the controllability canonical form as in \eqref{eq:canonical_single} necessarily guarantees uniform ensemble controllability of the whole ensemble, and vice versa. As a consequence, this gives rise to a canonical form for uniform ensemble controllability in the traditional sense that linear ensemble systems possess this form if and only if they are uniformly ensemble controllable. 

%=================== Theorem 3 =======================
\begin{theorem}[Ensemble Controllability Canonical Form]
\label{thm:canonical_ensemble}
Given a non-Sobolev type time-invariant linear ensemble system parameterized by $\b$ varying on a compact subset $K$ of $\mathbb{R}$ in the form of \eqref{eq:ensemble_linear}, that is,
$$\frac{d}{dt}x(t,\b)=A(\b)x(t,\b)+B(\b)u(t)$$
in which $x(t,\cdot)\in C(K,\mathbb{R}^n)$, $u(t)\in\mathbb{R}^m$, $A\in C(K,\mathbb{R}^{n\times n})$ with the eigenvalue functions $\l_1,\dots,\l_n\in C(K,\mathbb{R})$, and $B\in C(K,\mathbb{R}^{n\times m})$, and its diagonalized counterpart as in \eqref{eq:diagonal}, that is,
$$\frac{d}{dt}y(t,\b)=\Lambda(\b)y(t,\b)+\widetilde{B}(\b)u(t)$$
in which $\Lambda(\b)={\rm diag}(\l_1(\b),\dots,\l_n(\b))$, and $\widetilde{B}(\b)=P^{-1}(\b)B(\b)$ with $P(\b)\in C\big(K,{\rm GL}(n,\mathbb{R})\big)$ such that $P^{-1}(\b)A(\b)P(\b)$ is an upper triangular matrix with the diagonal entries $\l_1(\b),\dots,\l_n(\b)$. Then, the system is uniformly ensemble controllable on $C(K,\mathbb{R}^n)$ if and only if each individual system in the ensemble reparameterized by $\eta=(\eta_1,\dots,\eta_n)\in\l_1(K)\times\cdots\times\l_n(K)$ in \eqref{eq:reparameterization}, i.e., 
$$\frac{d}{dt}\left[\begin{array}{c}  z_1(t,\eta_1) \\ \vdots \\  z_n(t,\eta_n) \end{array}\right]=\left[\begin{array}{c} \eta_1 z_1(t,\eta_1) \\ \vdots \\ \eta_n z_n(t,\eta_n)\end{array}\right]+\left[\begin{array}{c} D_1(\eta_1) \\ \vdots \\ D_n(\eta_n) \end{array}\right]u(t),$$
where $D_i(\l_i)$ is Ensemble Controllability Criterion Matrix associated with the system of the $i^{\rm th}$ state 
$$\frac{d}{dt}z_i(t,\b)=\l_i(\b)z_i(t,\b)+\tilde{b}_i(\b)U(t)$$
of the system in \eqref{eq:diagonal} for every $i=1,\dots,n$, can be transformed to the form
\begin{align}
\label{eq:canonical_ensemble}
\frac{d}{dt}w(t,\eta)=C(\eta)w(t,\eta)+\overline B(\eta)u(t),
\end{align}
where $C(\eta)$ is the rational canonical form of $\Lambda(\eta)$, and
\begin{align*}
C(\eta)=\left[\begin{array}{ccc} C_1(\eta) & & \\ & \ddots & \\ & & C_{k(\eta)}(\eta) \end{array}\right]
\end{align*}
is the rational canonical from of $\Lambda(\eta)$ with the companion blocks $C_i(\eta)\in\mathbb{R}^{n_i(\eta)\times n_i(\eta)}$, and
%$\overline B(\eta)$ is in the block upper triangular form
 \begin{align*}
\overline B(\eta)=\left[\begin{array}{ccccc} \bar b_{11}(\eta) & \cdots & \bar b_{1k(\eta)}(\eta) & \cdots & \bar b_{1l(\eta)}(\eta) \\  & \ddots & \vdots & & \vdots \\  & & \bar b_{k(\eta)k(\eta)}(\eta) & \cdots & \tilde b_{k(\eta)l(\eta)}(\eta) \end{array}\right]
\end{align*}
is in the block upper triangular form with $b_{ii}(\eta)\in\mathbb{R}^{n_i(\eta)\times m_i(\eta)}$, and $n_i(\eta)$, $m_i(\eta)$, $k(\eta)$, and $l(\eta)$ are positive integers depending on $\eta$.
\end{theorem}
\begin{proof}
The proof directly follows from Proposition \ref{prop:separating_pts} and Theorem \ref{thm:canonical_single}.
\end{proof}

We would like to recapitulate that the ensemble controllability canonical form in \eqref{eq:canonical_ensemble}, compared with the one in \eqref{eq:single_input_ensemble_canonical}, is parameterized by the ``spectral parameter" $\eta=(\eta_1.\dots,\eta_k)\in\l_1(K)\times\cdots\times\l_n(K)$ instead of the original system parameter $\b\in K$, which is the key to representing uniform ensemble controllability in terms of the canonical form for classical controllability. However, treating the system in \eqref{eq:canonical_ensemble} as an ensemble system parameterized by $\eta$, it is by no means to imply its uniform ensemble controllability on $C\big(\l_1(K)\times\cdots\times\l_n(K),\mathbb{R}^n\big)$. Actually, this system may not even be well-defined on $C\big(\l_1(K)\times\cdots\times\l_n(K),\mathbb{R}^n\big)$ because its system matrix $C:\l_1(K)\times\cdots\times\l_n(K)\rightarrow\mathbb{R}^{n\times n}$ is not necessarily a continuous function in $\eta$. To see this, we note that the entries on the subdiagonal, i.e., the entries directly below the main diagonal, of $C$ may switch between 0 and 1 depending on the values of $\eta$, which leads to discontinuity of them, and hence $C$, as functions of $\eta$.

Following the consideration of $C$ as a matrix-valued function in the rational canonical form defined on $\l_1(K)\times\cdots\times\l_n(K)$, we refer to it as the \emph{functional canonical form} of $\Lambda$. Inheriting the separating point character from the classical rational canonical form, which separates repeated eigenvalues, the functional rational canonical form possesses the ability to separate the overlapping images of the eigenvalue functions. Specifically, the points $\eta_1,\dots,\eta_n\in\mathbb{R}$ in the images of the eigenvalue functions $\l_1$, $\dots$, $\l_n$ of $\Lambda$, respectively, are the eigenvalues of the functional rational canonical form $C(\eta)$ of $\Lambda$ evaluated at $\eta=(\eta_1,\dots,\eta_n)$. Together with the diagonal form of $\Lambda$, those in the shared spectra, i.e., with the same values, represented as the repeated eigenvalues of $C(\eta)$, must be distributed to different companion blocks as the simple eigenvalues. Then, denoting the set of eigenvalues of $C_i(\eta)$ by $\Lambda_i\subseteq\{\eta_1,\dots,\eta_n\}$ for each $i=1,\dots,k(\eta)$, we get a tower of sets ordered by inclusion as $\Lambda_1\supseteq\Lambda_2\supseteq\cdots\supseteq\Lambda_{k(\eta)}$ satisfying $|\Lambda_i|=n_i(\eta)$, the dimension of the $i^{\rm th}$ state $z_i(t,\eta)$ of the system in \eqref{eq:reparameterization},  for all $i=1,\dots,k(\eta)$. This observation immediately gives a more concrete representation of the ensemble controllability canonical form in \eqref{eq:canonical_ensemble} by representing the $i^{\rm th}$ companion block $C_i(\eta)$ of $C(\eta)$ and the $(n_1(\eta)+\dots+n_{i-1}(\eta)+1)^{\rm th}$ to $(n_1(\eta)+\dots+n_{i}(\eta)+1)^{\rm th}$ rows $\tilde b_{ii}(\eta)$, $\dots$, $\tilde b_{il(\eta)}(\eta)$ of $\overline B(\eta)$ as functions of the elements in $\Lambda_i$ as
\begin{align*}
C(\eta)=\left[\begin{array}{ccc} C_1(\Lambda_1) & & \\ & \ddots & \\ & & C_{k(\eta)}(\Lambda_{k(\eta)}) \end{array}\right]
\end{align*}
and
 \begin{align*}
& \overline B(\eta)=\nonumber\\
&\left[\begin{array}{ccccc} \tilde b_{11}(\Lambda_1) & \cdots & \tilde b_{1k(\eta)}(\Lambda_1) & \cdots & \tilde b_{1l(\eta)}(\Lambda_1) \\  & \ddots & \vdots & & \vdots \\  & & \tilde b_{k(\eta)k(\eta)}(\Lambda_{k(\eta)}) & \cdots & \tilde b_{k(\eta)l(\eta)}(\Lambda_{k(\eta)}) \end{array}\right].
\end{align*}

Parallel to the necessary classical controllability condition shown in Proposition \ref{prop:number_of_inputs}, it is also possible to derive a necessary uniform ensemble controllability condition in terms of the number of companion blocks in the functional canonical form of the system matrix of a linear ensemble system from its ensemble controllability canonical form.

%=========================== Corollary 2 ==========================
\begin{corollary}
\label{cor:canonical_ensemble}
Consider a non-Sobolev type time-invariant linear ensemble system defined on $C(K,\mathbb{R}^n)$ as in \eqref{eq:ensemble_linear}. If the system is uniformly ensemble controllable on $C(K,\mathbb{R}^n)$, then the number $m$ of control inputs applied to the system satisfies $m\geq\max_{\eta\in\prod_{i=1}^n\l_i(K)}k(\eta)$.
\end{corollary}
\begin{proof}
This corollary is a direct consequence of Theorem \ref{thm:canonical_ensemble} and Proposition \ref{prop:number_of_inputs}.
\end{proof}

This necessary ensemble controllability condition also has a separating point interpretation: the number of control inputs has to be greater than or equal to that of the shared spectra to guarantee ensemble controllability of linear systems, which again consolidates the importance of separating points in ensemble control theory. 

To conclude this paper, we revisit the previous examples to illustrate the use of the ensemble controllability canonical form in the study of linear ensemble systems.

%============================ Example 6 ============================
\begin{example}
\label{ex:canonical_ensemble}
Recall the system in \eqref{eq:uncontrollable_2d} defined on $C([1,2],\mathbb{R}^2)$ with the system and control matrices ${\rm diag}(\b,2\b)$ and $[1, 1]'$, respectively. Reparameterizing this system by the eigenvalue functions of its system matrix yields an ensemble indexed by $\eta=(\eta_1,\eta_2)=[1,2]\times[2,4]$ as
\begin{align}
\label{eq:reparam_uncontrollable}
\frac{d}{dt}z(t,\eta)=\left[\begin{array}{cc} \eta_1 & 0 \\ 0 & \eta_2 \end{array}\right]x(t,\b)+\left[\begin{array}{c} 1  \\ 1 \end{array}\right]u(t).
\end{align}
which also reveals the number of the shared spectra for the system matrix as
$$
k(\eta)=
\begin{cases}
1,\quad\text{if }\eta\neq(2,2),\\
2,\quad\text{if }\eta=(2,2).\\
\end{cases}
$$  
Then, uniform ensemble uncontrollability of the system is also indicated by the number of control inputs $m=1<2=\max_{\eta}k(\eta)$, according to Corollary \ref{cor:canonical_ensemble}. Equivalently, we can check that the individual system indexed by $(2,2)$ in the ensemble in \eqref{eq:reparam_uncontrollable} is not controllable on $\mathbb{R}^2$, which then disables the transformation of the original system in \eqref{eq:uncontrollable_2d} to the ensemble controllability canonical form as in \eqref{eq:canonical_ensemble}. Moreover, it also indicates that the form in \eqref{eq:canonical_ensemble} is indeed the correct ensemble controllability canonical form in the sense of characterizing uniform ensemble controllability in a necessary and sufficient fashion, in contrast to the canonical form in \eqref{eq:canonical_uncontrollable_2d} for single-input systems as shown in Example \ref{ex:canonical_uncontrollable_2d}.
%On the other hand, in Example \ref{ex:canonical_uncontrollable_2d}, we discovered that before the reparameterization, the system can still be globally transformed to the canonical form as in \eqref{eq:single_input_ensemble_canonical} despite its uncontrollability. However, the reparametrized ensemble in \eqref{eq:reparam_uncontrollable} does not admit a global controllability canonical form as in \eqref{eq:canonical_ensemble}, and the obstacle is uncontrollability of the individual indexed by $\eta=(2,2)$ on $\mathbb{R}^2$. 

On the other hand, if the system in \eqref{eq:uncontrollable_2d} is driven by one more control input, giving the system in \eqref{eq:controllable_2d}, then it becomes uniformly ensemble controllable as shown in Example \ref{ex:controllable_2d}. By Theorem \ref{thm:canonical_ensemble}, uniform ensemble controllability of the system in \eqref{eq:controllable_2d} can also be revealed by transforming its ``spectral parameterization''
\begin{align}
\label{eq:reparam_controllable}
\frac{d}{dt}z(t,\eta)=\left[\begin{array}{cc} \eta_1 & 0 \\ 0 & \eta_2 \end{array}\right]z(t,\b)+\left[\begin{array}{cc} 1 & 0  \\ 1 & 1 \end{array}\right]\left[\begin{array}{c} u(t)  \\ v(t) \end{array}\right]
\end{align}
to the ensemble controllability canonical form. To this end, for $\eta\neq(2,2)$, we pick 
\begin{align}
\label{eq:controllable_2d_reparam_canonical}
P(\eta)=\left[\begin{array}{cc} 1 & \eta_1 \\ 1 & \eta_2 \end{array}\right]
\end{align}
%whose first column is the control vector field of $u(t)$ and the second column is the product of the system matrix and the same vector field, 
and $\det(P(\eta))=\eta_2-\eta_1\neq0$ for all $\eta\in([1,2]\times[2,4])\backslash\{(2,2)\}$ implies its invertibility. Applying $P(\eta)$ to the system in \eqref{eq:controllable_2d_reparam_canonical} as a change of coordinates $w(t,\eta)=P^{-1}(\eta)z(t,\eta)$ yields
\begin{align}
\label{eq:controllable_2d_canonical}
\frac{d}{dt}w(t,\eta)=\left[\begin{array}{cc} 0 & -\eta_1\eta_2 \\ 1 & \eta_1+\eta_2 \end{array}\right]w(t,\eta)+\left[\begin{array}{cc} 1 &  \frac{-\eta_1}{\eta_2-\eta_1} \\ 0 & \frac{1}{\eta_2-\eta_1} \end{array}\right]\left[\begin{array}{c} u(t) \\ v(t) \end{array}\right],
\end{align}
whose system matrix is the function canonical form $C(\eta)$ of the system matrix ${\rm diag}(\b,2\b)$ of the system in \eqref{eq:controllable_2d} evaluated at $\eta\neq(2,2)$ and control matrix is in the (block) upper triangular form as desired. For the case $\eta=(2,2)$, the system matrix ${\rm diag}(\eta_1,\eta_2)={\rm diag}(2,2)$ is already in the functional rational canonical form, and hence the task is to transform the control matrix to an upper triangular form with the system matrix intact. In particular, we use the control matrix as the change of coordinates
$$P(2,2)=\left[\begin{array}{cc} 1 & 1 \\ 0 & 1 \end{array}\right],$$
which transforms the system in \eqref{eq:reparam_controllable} evaluated at $\eta=(2,2)$ to the desired form as
\begin{align}
\label{eq:controllable_2d_canonical_2}
\frac{d}{dt}w(t,2,2)=\left[\begin{array}{cc} 2 & 0 \\ 0 & 2 \end{array}\right]w(t,2,2)+\left[\begin{array}{cc} 1 &  0 \\ 0 & 1 \end{array}\right]\left[\begin{array}{c} u(t)  \\ v(t) \end{array}\right].
\end{align}
The systems in \eqref{eq:controllable_2d_canonical} and \eqref{eq:controllable_2d_canonical_2} then constitute the ensemble controllability canonical form of the system in \eqref{eq:controllable_2d}, and correspondingly, the system matrix 
\begin{align*}
C(\eta)=\begin{cases}
\ \left[\begin{array}{cc} 0 & -\eta_1\eta_2 \\ 1 & \eta_1+\eta_2 \end{array}\right],\quad\eta\in[1,2]\times[2,4]\backslash\{(2,2)\},\\
\ \left[\begin{array}{cc} 2 & 0 \\ 0 & 2 \end{array}\right],\quad\quad\quad\ \ \eta=(2,2)
\end{cases}
\end{align*}
of the ensemble canonical form gives the functional canonical form of the system matrix system matrix ${\rm diag}(\b,2\b)$ of the system in \eqref{eq:controllable_2d}. At last, we notice that the $(2,1)$-entry of $C(\eta)$ switches between 1 and 0, along with the $(1,1)$-entry switching between 0 and 2, so that $C:[1,2]\times[2,4]\rightarrow\mathbb{R}^{2\times 2}$ is not a continuous function in $\eta$ and hence the system in \eqref{eq:controllable_2d_canonical}, regarded as an ensemble system parameterized by $\eta\in[1,2]\times[2,4]$, cannot evolve on $C([1,2]\times[2,4],\mathbb{R}^2)$. This then provides a solid support for the observation below Theorem \ref{thm:canonical_ensemble} that the concept of uniform ensemble controllability is not well-defined for linear ensemble systems in the ``spectral parameterization". 
\end{example}

%Note that in the ensemble controllability canonical form \eqref{eq:controllable_2d_canonical} and \eqref{eq:controllable_2d_canonical_2} of the system in \eqref{eq:controllable_2d}, the control matrix is in the special form that its (block) diagonal elements are all column vectors with 1 in the first component and 0 elsewhere. This is because each companion block in the system matrix of the canonical form is generated by exactly one column of the control matrix of the reparameterized ensemble system in \eqref{eq:reparam_controllable}, and it happens frequently in the case that the number of control inputs is equal to the number of shared spectra as the system in \eqref{eq:controllable_2d}.

%======================== Conclusion ======================
\section{Conclusion}
The major contribution of this paper is the development of a holistic approach to algebraic characterizations of controllability properties for time-invariant linear systems, including classical controllability of finite-dimensional linear systems and uniform ensemble controllability of infinite-dimensional linear ensemble systems, by leveraging the technique of separating points. To lay the foundation, we interpret the dynamics of linear systems as the action of polynomials on vector spaces, then integrate the concept of separating points into the structure of finitely generated modules to establish a classical controllability canonical form for multi-input linear systems. In particular, for a system in this form, the system matrix is in the rational canonical form indicating the separation of the shared spectrum, which guarantees controllability. Based on this result, we develop the notion of the functional canonical form for matrix-valued functions by exploiting the method of ``spectral reparametrization" in the separating point technique, and then prove that it is necessary and sufficient for an uniformly ensemble controllable linear ensemble system to admit an ensemble controllability canonical form, in which the system and control matrix are in the functional canonical and block diagonal form, respectively. This, in turn, gives an algebraic characterization of the notion of separating points. It is worth noting that this work unifies the notion of separating points for classical and ensemble linear systems, which technically bridges the gap between classical and ensemble controllability and hence opens up the possibility of utilizing finite-dimensional methods, e.g., those in classical linear systems theory and matrix algebra, to study infinite-dimensional linear systems. Moreover, it also sheds light on the direction towards an inclusive ensemble control theory providing new ways of thinking about and effective tools for those control and learning problems of complex systems that greatly challenge us nowadays.

%%%%%%%%%%%%%%%%%%%%%%%%%%%%%%%%%%%
\bibliographystyle{ieeetr}
\footnotesize
\bibliography{Canonical_form_arXiv}

\end{document}